\documentclass[11 pt]{amsart}
\usepackage{hyperref}
\usepackage{changebar}
\usepackage{tikz}
\tikzset{>=latex}

\usepackage{latexsym}
\usepackage{amssymb, amsmath}
\usepackage{amsthm}
\usepackage{geometry}
\usepackage{comment}
\usepackage[T1]{fontenc}
\usepackage{lmodern}

\usepackage[stretch=10]{microtype}
\usepackage{tikz}
\usetikzlibrary{calc}
\usepackage{fixltx2e}

\numberwithin{equation}{section}

\newtheorem{theorem}{Theorem}[section]
\newtheorem{lemma}[theorem]{Lemma}
\newtheorem{cor}[theorem]{Corollary}

\newtheorem{proposition}[theorem]{Proposition}

\newtheorem{defin}[theorem]{Definition}

\newtheorem{remark}[theorem]{Remark}

\newcommand\cB{{\mathcal B}}

\newcommand{\cC}{\mathcal{C}}

\newcommand{\cG}{\mathcal{G}}

\newcommand{\cL}{\mathcal{L}}
\newcommand{\Lp}{\mathcal L}

\newcommand{\pa}{\mathcal{P}}

\newcommand{\cW}{\mathcal{W}}

\newcommand{\bV}{\mathbb{V}}

\newcommand{\bmu}{\mu_0}

\newcommand{\bvf}{\overline{\vf}}

\newcommand{\of}{\bar{f}}
\newcommand{\oM}{\overline{M}}
\newcommand{\oT}{\overline{T}}

\newcommand{\vf}{\varphi}

\newcommand{\ve}{\varepsilon}

\newcommand{\bo}{\mathbf{1}}

\def\beq{\begin{equation}}
\def\eeq{\end{equation}}

\begin{document}

\title{A Gentle Introduction to Anisotropic Banach Spaces}
\author{Mark F. Demers}
\address{Department of Mathematics, Fairfield University, Fairfield CT 06824, USA}
\email{mdemers@fairfield.edu}

\thanks{
The author was partly supported by NSF grants DMS 1362420 and 1800321.}

\date{\today}

\begin{abstract}
The use of anisotropic Banach spaces has provided a wealth of new results in the study of hyperbolic dynamical systems
in recent years, yet their application to specific systems is often technical and difficult to access.
The purpose of this note is to provide an introduction to the use of these spaces in the study of hyperbolic
maps and to highlight the important elements and how they work together.  This is done via a concrete example
of a family of dissipative Baker's transformations.  Along the way, we prove an original result connecting such
transformations with expanding maps via a continuous family of transfer operators acting on a single Banach space.  
\end{abstract}
\maketitle

\section{Introduction}
\label{sec:intro}

The study of anisotropic Banach spaces on which the transfer operator associated with a hyperbolic dynamical system has
good spectral properties has been
the subject of intense activity during the past 15 years.  Beginning with the seminal paper \cite{BKL}, there have been
a flurry of papers developing several distinct approaches, first for smooth uniformly hyperbolic maps 
\cite{Ba1, BaT, liv gouezel, liv gouezel2}, then piecewise hyperbolic maps \cite{demers liverani, bg1, bg2}, 
and finally to hyperbolic
maps with more general singularities, including many classes of billiards \cite{dz1, dz3} and their perturbations \cite{dz2}.
This technique has also been successfully applied to prove exponential rates of mixing for hyperbolic flows, a notoriously difficult problem,
following a similar trajectory:  first to contact Anosov flows \cite{Li04, Tsu1}, then to contact flows with discontinuities \cite{BaL}
and finally to billiard flows \cite{bdl}.

The purpose of this note is to provide a gentle introduction to the study of anisotropic Banach spaces via a concrete model:
a family of dissipative Baker's transformations.  This family of maps provides a prototypical hyperbolic setting and allows
for the application of transfer operator techniques without the technical difficulties associated with other concrete models,
such as dispersing billiards.  On the other hand, it avoids the full generality necessary for an axiomatic treatment of Anosov or 
Axiom A maps as found,
for example, in \cite{BaT, liv gouezel2}.
Despite its simplicity, the study of this family of maps includes all the essential elements needed for the successful application of 
this technique to more complex systems:  a suitable set of norms, the Lasota-Yorke inequalities required to prove
quasi-compactness of the transfer operator, a Perron-Frobenius argument for characterizing the peripheral spectrum, and the
approximation of distributions in the Banach space norm.  Thus we hope it will serve as an easily accessible
introduction to the subject for those who wish to pursue this mode of analysis in more complex systems.
This note is based, in part, on introductory lectures given at the June 2015 workshop DinAmici IV held in Corinaldo, Italy. 


\subsection{A brief survey of anisotropic spaces}

As mentioned above, there has been a wealth of activity in the application of anistropic Banach spaces to the study of
hyperbolic systems.  Although they vary in their application, they have one feature in common:  they all exploit the fact that
in hyperbolic systems, the transfer operator improves the regularity of densities along unstable manifolds, and its dual improves
the regularity of test functions along stable manifolds.  This differentiated treatment of stable and unstable directions is
what makes these spaces anisotropic.

In this section, we outline some of the principal branches of these efforts.  Roughly, they can
be divided into three groups:
\begin{enumerate}

  \item {\bf The geometric approach pioneered by Liverani and Gouezel \cite{liv gouezel}.}  
  This approach follows from the seminal work
  mentioned above by Blank, Keller and Liverani \cite{BKL}, who viewed the transfer operator as acting on distributions
  integrated against vector fields and various classes of test functions.  An essential difference introduced in \cite{liv gouezel}
  was to integrate against smooth test functions on stable manifolds only, or more generally {\em admissible stable curves} 
  whose tangent vectors lie in the
  stable cone, as opposed to integrating over the entire phase space, thus simplifying the application
  of the method.  In the smooth case, this technique was able to exploit higher smoothness of the map to obtain
  improved bounds on the essential spectral radius of the transfer operator \cite{liv gouezel}, and is applied to generalized
  potentials in \cite{liv gouezel2}.  It has also yielded significant results on dynamical determinants and zeta functions for both
  maps \cite{liv deter, liv tsujii} and flows \cite{liv poll giu}.
  
  For systems with discontinuities, integrating on stable curves greatly simplifies the geometric arguments
  required to control the growth in complexity due to discontinuities.  This was implemented first
  for two-dimensional piecewise hyperbolic maps (with bounded derivatives) \cite{demers liverani}, and then for various classes of billiards \cite{dz1, dz3}
  and their perturbations \cite{dz2}.  It has also led to the recent proof of exponential mixing for some billiard flows \cite{bdl}.
  This geometric approach has proved to be the most flexible so far in terms of the types of systems
  studied.
  
  \item  {\bf The Triebel-type spaces introduced by Baladi \cite{Ba1}.}  These spaces are based on the use of Fourier transforms
  to convert derivatives into multiplication operators.  They exploit the hyperbolicity of the map by
  taking negative fractional derivatives in the stable direction and positive derivatives in the unstable direction.
  Initially, the coordinates for these operators were tied to the invariant 
  dynamical foliations 
  associated with the hyperbolic systems, and these were assumed to be $C^\infty$
   \cite{Ba1}.  They were later generalized to include more general families of foliations 
   (only $C^{1+\epsilon}$ smooth and such that the family of foliations is invariant, not each foliation individually), and
  successfully applied to piecewise hyperbolic maps \cite{bg1, bg2}.  They were also the first norms succesfully adapted to contact flows with
  discontinuities in \cite{BaL}.
  
  \item  {\bf  The Sobolev-type ``microlocal spaces" of Baladi and Tsujii \cite{BaT}.}  In some sense,
  these spaces are an evolution of the Triebel-type spaces described above.  The spaces still exploit the hyperbolicity of the map
  by using Fourier transforms and pseudo-differential operators, taking negative derivatives in the stable direction and positive
  derivatives in the unstable direction.  Now, however, the invariant foliations are replaced by cones in the cotangent space on
  which these operators act, and the operators are averaged with respect to an $L^p$ norm.
  
  	Such spaces, and the semi-classical versions of Faure and coauthors \cite{faure roy stro}, have produced extremely strong
	results characterizing the spectrum of the transfer operator for smooth hyperbolic maps and flows 
	\cite{faure tsujii1, faure tsujii2, faure tsujii3}.  
	This approach has also achieved the sharpest bound to date for the essential spectral radius of the transfer operator
	via a variational formula \cite{BaT2}, and numerous results on dynamical determinants and zeta functions \cite{baladi book2}. 
	However, they have not been applied to systems with discontinuities,
	and it seems some new ideas will be needed to generalize them in this direction. \cite{baladi char} contains
	a recent attempt to develop this capability.
\end{enumerate}

This is only a brief description of the types of Banach spaces used to study hyperbolic systems in recent years, and is by no means
a comprehensive listing.  A more thorough and nuanced 
account is contained in the recent article of Baladi \cite{baladi review}. 
In addition, there are alternative approaches that use similar types of anisotropic constructions adapted to special cases.  The
recent work \cite{galatolo}, for example, constructs spaces using an averaged type of bounded variation, which is shown to be
effective for a class of partially hyperbolic maps with a skew-product structure.

In the present paper, we will follow the geometric approach of Liverani and Gou\"ezel described in (1) above.
These spaces are the most concrete of the types listed above, the integrals being taken on stable manifolds against suitable
test functions, and as such fit our purpose here best, which is to provide a hands-on introduction to the subject with
as few pre-requisities as possible.


\subsection{A pedagogical example}
\label{ped}

Before introducing the class of hyperbolic maps for which we will construct an appropriate anisotropic Banach space,
we consider the following simpler example\footnote{This example was communicated to the author by C.~Liverani 
some years ago and has served as inspiration ever since.} 
of a contracting map of the interval.

For expanding systems, the fact that the transfer operator increases the regularity of densities is well-understood.  It is
this feature which generally enables one to derive the Lasota-Yorke inequalities needed to prove its quasi-compactness
on a suitable space of functions compactly embedded in $L^1$ with respect to some reference measure. 
Although $L^1$ is in general both too small and too large a space in the context of hyperbolic maps, 
it is instructive to see that similar inequalities can be derived in the purely contracting case as well, and to note that the
transfer operator in this case increases the regularity of {\em distributions}.

Let $I = [0,1]$, and $T : I \circlearrowleft$ be a $C^1$ map satisfying $|T'(x)| \le \lambda < 1$ for all $x \in I$.  It is a well-known
consequence of the contraction mapping theorem that there exists a unique $a \in I$ such that $T(a) = a$. 

For $\alpha \in (0,1]$, let $C^\alpha$ denote the set of H\"older continuous functions on $I$ with exponent $\alpha$.  
For $\vf \in C^\alpha$, define 
\begin{equation}
\label{eq:holder def}
H^\alpha(\vf) = \sup_{\substack{x,y \in I \\ x \neq y}} \frac{|\vf(x) - \vf(y)|}{|x-y|^\alpha}, \qquad \mbox{ and } 
|\vf|_{C^\alpha} = |\vf|_{C^0} + H^\alpha(\vf),
\end{equation}
where $|\vf|_{C^0} = \sup_{x \in I} |\vf(x)|$.  Since $T$ is $C^1$, if $\vf \in C^\alpha$, then $\vf \circ T \in C^\alpha$.

Let $(C^\alpha)^*$ be the dual of $C^\alpha$.  For a distribution $\mu \in (C^\alpha)^*$, we define the action of the transfer
operator $\Lp$ associated with $T$ via its dual,
\[
\Lp \mu(\vf) = \mu(\vf \circ T), \qquad \mbox{for all $\vf \in C^\alpha$}.
\]
Thus $\Lp\mu \in (C^\alpha)^*$ as well.    

For $\mu \in (C^\alpha)^*$, define
\[
\| \mu \|_\alpha = \sup_{\substack{\vf \in C^\alpha \\ |\vf|_{C^\alpha} \le 1}} \mu(\vf)  .
\]
The reader can check that $\| \cdot \|_\alpha$ satisfies the triangle inequality, and is a norm.  

If $f \in C^1(I)$, then we can identify $f$ with the measure $d\mu = f dm$, where
$m$ denotes Lebesgue measure on $I$.  With this identification, $C^1 \subset (C^\alpha)^*$. 
When we regard $f \in C^1$ as an element of $(C^\alpha)^*$, we will write $f(\vf) = \int_I f \vf \, dm$.

We define $\cB^\alpha$ to be the space $(C^\alpha)^*$ equipped with the $\| \cdot \|_\alpha$ 
norm.
Note that $\cB^\alpha$ is a Banach space of distributions that includes all Borel measures on $I$;
this includes the point mass at $a$, $\delta_a$.

We will work with the spaces $(\cB^\alpha, \| \cdot \|_\alpha)$ and $(\cB^1, \| \cdot \|_1)$, the latter of which
has the same definitions as $\cB^\alpha$, but with $\alpha = 1$.

For $\vf \in C^\alpha$, $n \ge 0$, define $\bvf_n = \int_I \vf \circ T^n \, dm$, recalling that $m$ denotes Lebesgue measure.

Now fix $\alpha < 1$.  For $\vf \in C^\alpha$ and $n \ge 0$, we estimate for $\mu \in \cB^\alpha$,
\begin{equation}
\label{eq:contract ly}
 \Lp^n \mu(\vf)  = \mu(\vf \circ T^n - \bvf_n) + \mu(\bvf_n) \le \| \mu \|_\alpha |\vf \circ T^n - \bvf_n|_{C^\alpha} + \| \mu \|_1 |\bvf_n|_{C^1}.
\end{equation}
Since $\vf \circ T^n \in C^\alpha$, there exists $u \in I$ such that $\vf \circ T^n(u) = \bvf_n$.  Thus for $x, y \in I$, we have
\begin{equation}
\label{eq:avg subtract}
\begin{split}
& |\vf \circ T^n(x) - \bvf_n| = |\vf \circ T^n(x) - \vf \circ T^n(u)| \le |\vf|_{C^\alpha} |T^n(x) - T^n(u)|^\alpha
\le \lambda^{n \alpha} |\vf|_{C^\alpha} \, , \\ 
& |\vf \circ T^n(x) - \bvf_n - \vf \circ T^n(y) + \bvf_n| \le |\vf|_{C^\alpha} \lambda^{n \alpha} |x-y|^\alpha .
\end{split}
\end{equation}
Thus $|\vf \circ T^n - \bvf|_{C^\alpha} \le \lambda^{\alpha n} |\vf|_{C^\alpha}$, and $|\bvf_n|_{C^1} \le |\vf|_{C^0}$.
Using these estimates in \eqref{eq:contract ly} and taking the supremum over $\vf \in C^\alpha$ with $|\vf|_{C^\alpha} \le 1$ yields,
\begin{equation}
\label{eq:ly 2}
\| \Lp^n \mu \|_{\alpha} \le \lambda^{\alpha n} \| \mu \|_\alpha + \| \mu \|_1 .
\end{equation}
A simpler estimate (without subtracting the average in \eqref{eq:contract ly}) shows that $\| \Lp^n \mu \|_1 \le \| \mu \|_1$.

Since the unit ball of $C^1$ is compactly embedded in $C^\alpha$, it follows by duality that the unit ball of
$\cB^\alpha$ is compactly embedded in $\cB^1$.  Thus these inequalities constitute a standard set of
Lasota-Yorke inequalities for $\Lp$ acting on $\cB^\alpha$.  Indeed, Hennion's argument \cite{hennion spec}
using the Nussbaum formula \cite{nussbaum} for the essential spectral radius implies that the essential spectral radius
is at most $\lambda^\alpha$.
Since $\delta_a$ belongs to $\cB^\alpha$
and satisfies $\Lp \delta_a = \delta_a$, we conclude that $\Lp$ is quasicompact as an operator on 
$\cB^\alpha$:
the spectral radius of $\Lp$ is 1, and the part of the
spectrum outside any disk of radius $> \lambda^\alpha$ is finite-dimensional \cite{hennion}.

Indeed, we can say much more about the peripheral spectrum.  Since $\| \Lp^n \mu \|_\alpha \le \| \mu \|_\alpha$, no element of
the peripheral spectrum contains any Jordan blocks.  Thus we have the following decomposition
of $\cL$.  There exist $N \in \mathbb{N}$, $\theta_j \in [0,1)$, and operators $\Pi_j$, $R : \cB^\alpha \circlearrowleft$, satisfying
$\Pi_j^2 = \Pi_j$, $\Pi_j R = R \Pi_j = 0$ and $\Pi_k \Pi_j = \Pi_j \Pi_k = 0$ for $k \neq j$, such that
\begin{equation}
\label{eq:contract decomp}
\Lp = \sum_{j=0}^N e^{2\pi i \theta_j} \Pi_j + R,
\end{equation}
and $R$ satisfies $\| R^n \|_\alpha \le Cr^n$ for some $C>0$, $r<1$, and all $n \ge 0$.  By convention we consider $\theta_0 = 0$,
so $\Pi_0$ denotes the projection onto the subspace of invariant distributions in $\cB^\alpha$.
Let $\bV_j := \Pi_j \cB^\alpha$.

To characterize the $\theta_j$, it is convenient to note that $\Lp$ is the dual of the Koopman 
operator $K\vf = \vf \circ T$ acting on $\vf \in C^\alpha$.  Indeed, by calculations similar to \eqref{eq:avg subtract} above, it follows that
\[
|K^n \vf |_{C^\alpha} \le \lambda^{\alpha n} |\vf |_{C^\alpha} + |\vf|_{C^0}, \quad \mbox{and}
\quad |K^n \vf |_{C^0} \le |\vf|_{C^0},
\]
thus $K$ is quasi-compact as an operator on $C^\alpha$, and its eigenspaces on the unit circle
are finite dimensional.
By \cite[Theorem VI.7]{reed}, the spectra of $K$ and $\Lp$ are the same so their eigenvalues on the unit circle coincide. 

Since $C^\alpha$ is closed under products,
and $K(\vf_1 \vf_2) = (K\vf_1)(K\vf_2)$, 
it follows that if $K\vf_1 = e^{2\pi i \theta_j} \vf_1$ and $K \vf_2 = e^{2\pi i \theta_k} \vf_2$, then
$K(\vf_1\vf_2) = e^{2\pi i (\theta_j + \theta_k)} \vf_1 \vf_2$, so that the peripheral spectrum of $K$,
and thus of $\Lp$, forms a group.  By quasi-compactness, the group must be finite, forcing 
$\theta_j \in \mathbb{Q}$ for all $0 \le j \le N$. 


Now suppose $\mu \in \bV_j$.  Then
for any $\vf \in C^\alpha$ and $n \ge 0$,
\[
|\mu(\vf)| = |\Lp^n \mu (\vf)| = | \mu(\vf \circ T^n) |
\le \| \mu \|_\alpha ( |\vf \circ T^n|_{C^0} + H^\alpha(\vf \circ T^n) ) 
\le \| \mu \|_\alpha ( |\vf |_{C^0} + \lambda^{n \alpha} H^\alpha(\vf)),
\]
by \eqref{eq:avg subtract}.  Taking the limit as $n \to \infty$ yields, $| \mu(\vf) | \le \| \mu \|_\alpha |\vf|_{C^0}$,
so that $\mu$ extends to a bounded linear functional on $C^0(I)$, i.e. it is a Radon measure.\footnote{Recall that
    a Radon measure is a measure that is both outer and inner regular and locally finite.} 

Thus each $\bV_j$ is comprised entirely of measures.  Since $\theta_j \in \mathbb{Q}$, there exists
$q_j \in \mathbb{N}$ such that for each $\mu \in \bV_j$, $\Lp^{q_k} \mu = \mu$, so that 
$\mu$ is an invariant measure for $T^{q_j}$.  But since $T^n$ is a $C^1$ contracting map for all
$n \ge 1$, $\delta_a$ is its only invariant probability measure.  Thus $\mu = c \delta_a$, for some 
$c \in \mathbb{R}$.

We have proved the following theorem. 

\begin{theorem}
\label{thm:contract}
$\Lp$ acting on $\cB^\alpha$ has a spectral gap:  1 is a simple eigenvalue
and the rest of the spectrum is contained in a disk of radius $r<1$.  Thus $\delta_a$ is
the only invariant distribution and $\Lp^n \mu$ converges at an exponential rate to $\delta_a$
for any $\mu \in \cB^\alpha$ with $\mu(1) = 1$.
\end{theorem}

\begin{remark}
An important idea that emerges from this example is that for a contracting map, it is not sufficient
to consider the space of measures on $I$, i.e. the dual of $C^0(I)$, if we expect the transfer
operator to increase regularity.  This is clear from considering \eqref{eq:ly 2} with $\alpha = 0$,
which provides no contraction.  Thus we must look in a larger space of distributions.

This is the main idea motivating the norms that we shall define for hyperbolic maps:  we will
construct spaces of distributions by integrating against H\"older continuous test functions
along stable manifolds.
\end{remark}


\section{Setting and Definition of the Banach Spaces}

In this section we define the class of hyperbolic maps with which we shall work.  Before beginning, it is helpful to outline the key elements of Section~\ref{ped}
that we will recreate in the hyperbolic setting.

\begin{itemize}
	\item[(i)] Two Banach spaces $(\cB, \| \cdot \|)$ and $(\cB_w, | \cdot |_w)$ (the strong and the weak)
	on which the transfer 
	operator $\Lp$ acts continuously.
	\item[(ii)] Domination of the weak norm by the strong norm, $| \cdot |_w \le \| \cdot \|$. 
	\item[(iii)]  The compactness of the unit ball of $\cB$ in $\cB_w$.
	\item[(iv)] Lasota-Yorke type inequalities: There exist $C \ge 0$, $\rho < 1$ such that for all $n \ge 0$,
	\[
	\| \Lp^n f \| \le C \rho^n \| f\| + C |f|_w,   \qquad | \Lp^n f |_w \le C |f |_w, \qquad \forall f \in \cB .
	\]
	\item[(v)] Existence of a conformal measure $m$ such that $\Lp^* m = m$. 
	\item[(vi)]  A Perron-Frobenius type argument to characterize the peripheral spectrum of $\Lp$. 
\end{itemize}

These are the essential ingredients needed in the application of the method, which are common to
more complicated classes of maps.  They also establish a sufficiently robust foundation from which
to apply various forms of perturbation theory \cite{demers liverani, dz2}.  We will show an
application of this to a degenerate limit in Section~\ref{sec:sing}.


\subsection{A family of Baker's transformations}
\label{bake}

In this section, we describe the family of Baker's transformations for which we will construct 
our Banach spaces.  Let
$M = [0,1]^2$ be the unit square in $\mathbb{R}^2$.  Fix $\kappa \in \mathbb{N}$, $\kappa \ge 2$, and $\lambda \in \mathbb{R}$ such that 
$0 < \lambda \le 1/\kappa$.  We describe a generalized $(\kappa, \lambda)$ Baker's transformation geometrically 
as follows:  We expand $M$ in the horizontal
direction by the factor $\kappa$ and contract $M$ in the vertical direction by the factor $\lambda$.  The resulting
rectangle is then split into $\kappa$ equal rectangles of length 1, and those rectangles are mapped isomorphically onto 
disjoint horizontal rectangles in $M$.  See Figure~\ref{fig:map}.

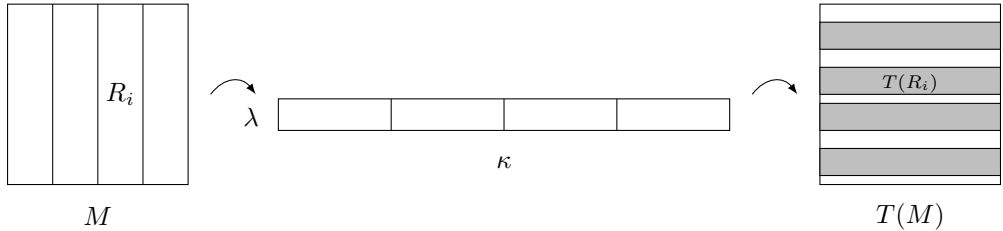
\begin{figure}[h]
\begin{tikzpicture}[x=6mm,y=6mm]
\draw (0,0) rectangle (4,4) ;

\draw (1,0) -- (1,4);
\draw (2,0) -- (2,4);
\draw (3,0) -- (3,4);

\node at (2.5, 2) {\small$R_i$};

 \node at (2,-.7){\small$M$};
  
 \draw[->] (4.5,2) .. controls (4.8,2.4) and (5.2,2.4) .. (5.5,2);

 \node at (5.4,1.5){\small$\lambda$};
  
  \draw (6,1.2) rectangle (16,1.9);
  \draw (8.5,1.2) -- (8.5,1.9);
  \draw (11,1.2) -- (11,1.9);
  \draw (13.5,1.2) -- (13.5,1.9);
  
   \node at (11,.5){\small$\kappa$};
  
 \draw[->] (16.5,2) .. controls (16.8,2.4) and (17.2,2.4) .. (17.5,2);
    
 \draw (18,0) rectangle (22,4);
 
 \filldraw[fill=black!25!white, draw=black] (18,.2) rectangle (22,.8);
\filldraw[fill=black!25!white, draw=black] (18,1.2) rectangle (22,1.8);
\filldraw[fill=black!25!white, draw=black] (18,2) rectangle (22,2.6);
\filldraw[fill=black!25!white, draw=black] (18,3) rectangle (22,3.6);

\node at (20,-.7){\small$T(M)$};
\node at (20,2.25){\tiny$T(R_i)$};
 \end{tikzpicture}
\caption{The map $T = T_{\kappa, \lambda}$ with $\kappa = 4$ and $\lambda < 1/4$.  The image $T(M)$
is shown in grey on the right.}
\label{fig:map}
\end{figure}

An equivalent definition is to subdivide $M$ into $\kappa$ vertical rectangles of width $1/\kappa$, 
denoted $R_i$.  The 
Baker's transformation $T = T_{\kappa, \lambda} : M \circlearrowleft$ is defined by requiring the action of $T|_{R_i}$ to be an
affine expansion by $\kappa$ in the horizontal direction and contraction by $\lambda$ in the vertical direction, such that
the sets $T(R_i)$, $i = 1,\ldots \kappa$, are pairwise disjoint (mod 0 with respect to Lebesgue measure).

This construction leaves us free to choose how to order the images $T(R_i)$ in $M$, 
and also in what orientation to place
them.  We do not specify these choices since they do not impact our analysis in any way.  All our results will apply to any choice
of $(\kappa, \lambda)$ Baker's transformation as defined above.

Note that if $\lambda = 1/\kappa$, then $T_{\kappa, \lambda}$ preserves Lebesgue measure on $M$.   
Otherwise, $T_{\kappa, \lambda}$
is dissipative.  We remark that for $\lambda < 1/\kappa$, $T_{\kappa, \lambda}$ is not a piecewise diffeomorphism of
$M$, i.e. $\cup_{i=1}^\kappa T_{\kappa, \lambda}(R_i) \subsetneq M$, so that formally
the general frameworks constructed in \cite{demers liverani} and \cite{dz3} for hyperbolic maps with singularities
do not apply, although their assumptions\footnote{\cite{dz3} has the further problem that (H1) of that paper is not satisfied in the present
setting.  (H1) would require that the Jacobian of $T_{\kappa,\lambda}$, which equals $\kappa \lambda$, is large compared
to the contraction obtained in the Lasota-Yorke inequalities, i.e. from Prop.~\ref{prop:ly}, $\kappa \lambda \ge \max \{ \lambda^\alpha, \kappa^{-\beta}\}$, which clearly fails when $\lambda$ is much smaller than $\kappa^{-1}$.}
on hyperbolicity and control of singularity sets are satisfied.
As we shall see, however, the method of those papers is readily implemented in this setting.

\begin{remark}
\label{rem:dist}
The requirement that $T_{\kappa,\lambda}$ be piecewise affine is clearly not necessary to implement any of the
techniques in this paper.  One could allow nonlinearities for $T$, as long as the key property of bounded distortion
were maintained.  That is:  Assume $x, y$ lie on the same stable manifold $W$ of $T$.  Then there exists $C>0$ such that
for all $n \ge 0$, $\log \frac{J^s_W T^n(x)}{J^s_W T^n(y)} \le C d(x,y)^p$ for some $p>0$, where $J^s_W T^n$ is the 
(stable) Jacobian of $T^n$ along $W$.
Since the inclusion of nonlinear effects will do nothing to illuminate the present technique, we omit it.
\end{remark}


\subsection{Transfer operator}

It is clear from the definition of $T = T_{\kappa, \lambda}$ that $T$ maps vertical lines into vertical lines
and horizontal lines to unions of horizontal lines.  We denote by $\cW^s$ ($\cW^u$) the set of full vertical (horizontal) line
segments of length 1 in $M$.
These sets comprise a collection of local stable and unstable manifolds, respectively.

We will work with classes of H\"older continuous functions $C^\alpha$, $\alpha \in [0,1]$.  For simplicity, when we
write $C^1$, we mean $C^\alpha$ with $\alpha =1$, which is simply Lipschitz.  

For $\alpha \in [0,1]$, define a H\"older norm along stable leaves as follows.  For any bounded function $\vf$,
\[
| \vf |_{C^\alpha(\cW^s)} = \sup_{W \in \cW^s} | \vf|_{C^\alpha(W)},
\quad \mbox{where}  \quad | \vf|_{C^\alpha(W)} = \sup_{x \in W} |\vf(x)| + H^\alpha_W(\vf),
\]
and $H^\alpha_W(\vf)$ is the H\"older constant along $W$ with exponent $\alpha$, as defined
in \eqref{eq:holder def}.
Let $C^1(\cW^s)$ denote the set of functions that are Lipschitz along stable leaves, i.e.
measurable functions $\vf$ with $| \vf |_{C^1(\cW^s)} < \infty$.  For $\alpha < 1$, define $C^\alpha(\cW^s)$ to be the
completion of $C^1(\cW^s)$ in the $| \cdot |_{C^\alpha(\cW^s)}$ norm.\footnote{We could simply define $C^\alpha(\cW^s)$ to be the
set of functions $\vf$ such that $| \vf |_{C^\alpha(\cW^s)} < \infty$, but we find the slightly smaller space convenient for
establishing the injectivity of the embedding $\cB \hookrightarrow \cB_w$ in Lemma~\ref{lem:embeddings}.  With this definition,
$C^\alpha(\cW^s)$ includes all functions $\vf$ such that $| \vf |_{C^{\alpha'}(\cW^s)} < \infty$ for some $\alpha' > \alpha$.}
Similarly, for each $W \in \cW^s$, $C^\alpha(W)$ denotes the completion of $C^1(W)$ in the $| \cdot |_{C^\alpha(W)}$ norm.
The reader can check that $C^\alpha(\cW^s)$ is a Banach space for all $\alpha \in [0,1]$.
We define $C^\alpha(\cW^u)$ similarly.

If $\vf \in C^\alpha(\cW^s)$, then $\vf \circ T \in C^\alpha(\cW^s)$.  This is because $T$ is $C^1$ on vertical lines, and for any
$W \in \cW^s$, $T^{-n}W$ is a union of elements of $\cW^s$ for all $n \ge 1$.  This allows us to define the
transfer operator $\Lp = \Lp_{\kappa, \lambda}$ associated with $T = T_{\kappa, \lambda}$ on $(C^\alpha(\cW^s))^*$
as follows,
\begin{equation}
\label{eq:dist def}
\Lp f (\vf) = f (\vf \circ T), \qquad \mbox{for all $\vf \in C^\alpha(\cW^s), \; f \in (C^\alpha(\cW^s))^*$.}
\end{equation}
Denote Lebesgue measure on $M$ by $m$.  We identify $f \in C^1(\cW^u)$ with the measure
$d\mu_f = f dm$.  With this identification, $f \in (C^\alpha(\cW^s))^*$, and
$\Lp f$ is associated with the measure having density
\begin{equation}
\label{eq:trans def}
\Lp f(x) = \frac{f \circ T^{-1}(x)}{JT(T^{-1}x)} = \frac{f \circ T^{-1}(x)}{\kappa \lambda}, 
\end{equation}
with respect to $m$.  Here, $JT$ denotes the Jacobian of $T$ with respect to $m$, which is simply
$\kappa \lambda$.  This is the familiar pointwise formula for the transfer operator.  With this pointwise definition,
$\Lp f = 0$ on $M \setminus TM$.  In addition, $m$ is a conformal measure for $\Lp$, i.e. $\Lp^*m = m$.

We will use the notation $f(\vf) = \int_M f \, \vf \, dm$ throughout the paper.  We will also denote
by $\bo$ the constant function with value 1 on $M$.



\subsection{Definition of the Banach spaces}
Considered one stable leaf at a time, the action of $T$ is very similar to that of the contracting map discussed in 
Section~\ref{ped}.  The weak norm defined below will play the analogous role to $\| \cdot \|_1$ from Section~\ref{ped}.
However, since the Baker's map is hyperbolic, the strong norm will have two components:  a {\em strong stable norm}, which plays
the role analogous to that of $\| \cdot \|_\alpha$ from Section~\ref{ped}, considering the map one stable leaf at a time; and
a {\em strong unstable norm}, which allows us to compare integrals along different stable leaves.  The strong unstable norm has
no analogue in Section~\ref{ped} and
exploits the fact that $\Lp$ increases regularity of functions along unstable leaves.

For $f \in C^1(\cW^u)$, define the weak norm of $f$ by
\[
| f |_w = \sup_{W \in \cW^s} \sup_{\substack{ \vf \in C^1(W) \\ |\vf|_{C^1(W)} \le 1}} 
\int_W f \, \vf \, dm_W,
\]
where $m_W$ is the arclength measure on $W$.

Next, for $W \in \cW^s$, we define the following coordinates,
\begin{equation}
\label{eq:coords}
W = \{ (s,t) \in M : s = s_W, t \in [0,1] \},
\end{equation}
where $s_W$ is the horizontal coordinate of the vertical segment $W$.
With these coordinates, we can define the distance between two curves
$W_1, W_2 \in \cW^s$, to be $d(W_1, W_2) = |s_{W_1} - s_{W_2}|$.  Moreover, for 
two test functions $\vf_i \in C^1(W_i)$, $i =1, 2$, we define the distance between
them by
\[
d_0(\vf_1, \vf_2) = \sup_{t \in [0,1]} |\vf_1(s_{W_1}, t) - \vf_2(s_{W_2}, t)|.
\]

Now fix $\alpha, \beta \in (0,1)$ with\footnote{The restriction $\beta \le 1- \alpha$ is used only
in the proof of Lemma~\ref{lem:mult}.} $\beta \le 1- \alpha$.  Define the strong stable norm of $f$ by
\[
\| f \|_s = \sup_{W \in \cW^s} \sup_{\substack{ \vf \in C^\alpha(W) \\ |\vf|_{C^\alpha(W)} \le 1}} 
\int_W f \, \vf \, dm_W.
\]
Finally, define the strong unstable norm\footnote{An alternative strong unstable norm is
$\displaystyle \| f \|_u = \sup_{W \in \cW^s}  \sup_{\substack{ \vf \in C^1(W) \\ |\vf|_{C^1(W)} \le 1}}  \int_W (d^u f) \, \vf \, dm_W$, where $d^u f$ is the derivative of $f$ in the unstable direction; however, this does not work well in more general systems when the distortion of $DT$ is only
H\"older continuous, as in the case of billiards, or when the unstable foliation is less regular.  So we adopt the definition of $\| f \|_u$ that is
analogous to that used in other works with low regularity \cite{demers liverani, dz1}.} of $f$ by
\[
\| f \|_u = \sup_{W_1, W_2 \in \cW^s} \sup_{\substack{\vf_i \in C^1(W_i) \\ |\vf_i |_{C^1(W_i)} \le 1 \\ d_0(\vf_1, \vf_2) = 0}}
d(W_1, W_2)^{-\beta} \left| \int_{W_1} f \, \vf_1 \, dm_{W_1} - \int_{W_2} f \, \vf_2 \, dm_{W_2} \right|  \, .
\]
Finally, the strong norm of $f$ is defined by $\| f \|_{\cB} := \| f \|_s + \| f \|_u$.

\begin{defin}
We define the weak space $\cB_w$ to be the completion of $C^1(\cW^u)$ in the weak norm, and
the strong space $\cB$ to be the completion of $C^1(\cW^u)$ in the strong norm.
\end{defin}

\begin{remark}
\label{rem:pairs}
One may view the norms $| \cdot |_w$ and $\| \cdot \|_s$ integrating against test functions on stable leaves as dual
in spirit to the notion of {\em standard pairs}, which have been used extensively in hyperbolic
dynamics.  Originally formulated by Dolgopyat \cite{dolgo contract} as a way to 
generalize the coupling method
introduced by Young \cite{young poly}, standard pairs have evolved to be one of the principal
tools in the analysis of the statistical properties of nonuniformly hyperbolic dynamical systems
(see \cite[Section 7.3]{chernov book} for an historical overview).  
\end{remark}


\subsection{Statement of results for $\lambda > 0$}

In this section, we summarize the principal results of the paper for the 
class of Baker's transformations $T_{\kappa, \lambda}$.

Although our spaces are defined abstractly as closures of $C^1(\cW^u)$ in the strong and
weak norms, the following lemma connects our spaces with distributions of order one on $M$,
and provides useful relations to more familiar function spaces.

\begin{lemma}
\label{lem:embeddings}
For any $\beta' \in (\beta, 1]$, we have the following sequence of continuous embeddings,
\[
C^1(M) \hookrightarrow C^{\beta'}(\cW^u) \hookrightarrow \cB \hookrightarrow \cB_w \hookrightarrow (C^1(\cW^s))^*.
\]
All the embeddings are injective except\footnote{Note, however, that the embedding $C^1(M) \hookrightarrow \cB$ is
injective.} $C^{\beta'}(\cW^u) \hookrightarrow \cB$.
Moreover, the embedding $\cB \hookrightarrow \cB_w$ is relatively compact.
\end{lemma}

\noindent
The proof of Lemma~\ref{lem:embeddings} appears in Section~\ref{props}.

Choose $\kappa \ge 2$ and $\lambda \in (0,\kappa^{-1}]$, and fix a Baker's map $T_{\kappa, \lambda}$.
The main theorem characterizing the spectrum of $\Lp  = \Lp_{\kappa, \lambda}$ on $\cB$ is the following.

\begin{theorem}
\label{thm:spec}
$\Lp$ is quasi-compact as an operator on $\cB$:  Its spectral radius is 1 and its essential
spectral radius is bounded by $\rho = \max \{ \lambda^\alpha, \kappa^{-\beta} \} < 1$.
Moreover, 
\begin{itemize}
  \item[a)]  $\Lp$ has a spectral gap:  the eigenvalue 1 is simple, and all other eigenvalues have
modulus strictly less than 1.
  \item[b)]  The unique element $\bmu$ of $\cB$ such that $\Lp \bmu = \bmu$ with $\bmu(\bo) =1$ is a measure.
It is the unique physical measure\footnote{An invariant measure $\mu$ is said to be a physical measure for $T$ if there
exists a positive Lebesgue measure set $B \subseteq M$ such that 
$\frac 1n \sum_{k=0}^{n-1} \psi \circ T^k(x)  \to \mu(\psi)$ for all $x \in B$ and all $\psi \in C^0(M)$.} for the system and its conditional measures on unstable leaves are equal to arclength.\footnote{This
implies in particular that $\mu_0$ is the unique SRB measure for the system.  See \cite{young SRB} for a discussion of
the importance of such measures.}
  \item[c)] There exists $\sigma < 1$ and $C>0$ such that for any $f \in \cB$
satisfying\footnote{Here $f(\bo)$ denotes the distribution $f$ applied to the constant
function $\bo$.} $f(\bo) = 1$, we have
\[
\| \Lp^n f - \bmu \|_{\cB} \le C \sigma^n \| f \|_{\cB} \, , \qquad \mbox{for all $n \ge 0$.}
\]
\end{itemize}
\end{theorem}

\noindent
Theorem~\ref{thm:spec} is proved in Section~\ref{periph}.

\subsubsection{Limit theorems}
\label{limit theorems}

The power of the present method is demonstrated by the fact that once
a spectral gap for $\Lp$ has been established, it in turn implies many limit theorems and statistical properties for $T$ 
using standard arguments. 
These limit theorems have been the subject of many recent studies via a variety of techniques,
including the use of Young towers
\cite{hennion, melbourne nicol, chazottes, rey bellet, gouezel}.  The present framework, however,
provides strengthened results in some cases, since it allows one to derive limit theorems with respect to
non-invariant measures as well as invariant ones.

We state several limit theorems here and show how they follow from our functional
analytic framework in Section~\ref{limit proofs}.

Throughout this section, $g$ denotes a real-valued function in $C^1(M)$,
 and
$S_ng = \sum_{j=0}^{n-1} g \circ T^j$.

\medskip
\noindent
{\em Large deviation estimates.}
Large deviation estimates typically take the form
\[
\lim_{\ve \to 0} \lim_{n \to \infty} \frac 1n \log \bmu \left( x \in M : \frac 1n S_ng(x) \in [t - \ve, t+\ve] \right) = - I(t),
\]
where $I(t) \ge 0$ is called the rate function.  When $I(t)>0$, such estimates provide exponential bounds on the rate of convergence
of $\frac 1n S_n g$ to $\bmu(g)$.  More generally, one can ask about the above limit
when $\bmu$ is replaced by a non-invariant measure $\nu$, for example Lebesgue measure or some other reference measure.
In the present context, we prove a large deviation estimate for all
probability measures $\nu \in \cB$ with the same rate function $I$.

\medskip
\noindent
{\em Central Limit Theorem.}
Assume $\bmu(g)=0$ and let $(g \circ T^j)_{j\in \mathbb{N}}$ be 
a sequence of  random variables on the probability space $(M, \nu)$, where $\nu$ is
a (not necessarily invariant) probability measure on the Borel $\sigma$-algebra.
We say that the triple $(g, T, \nu)$ satisfies a Central Limit Theorem if there exists
a constant $\varsigma^2 \geq 0$ such that
\[
\frac{S_ng}{\sqrt{n}} \; \; \stackrel{\mbox{dist.}}{\longrightarrow} \; \; \mathcal{N}(0, \varsigma^2),
\]
where $\mathcal{N}(0, \varsigma^2)$ denotes the normal distribution with mean 0 and
variance $\varsigma^2$.

\medskip
\noindent
{\em Almost-sure Invariance Principle.}
Assume again that $\bmu(g)=0$ and as above distribute $(g \circ T^j)_{j\in \mathbb{N}}$ according
to a probability measure $\nu$.
Suppose there exists $\ve >0$, a probability space $\Omega$ with a sequence of random variables $\{ X_n \}$
satisfying $S_ng \stackrel{\mbox{dist.}}{=} X_n$ and a
Brownian motion $W$ with variance $\varsigma^2 \ge 0$ such that
\[
X_n = W(n) + \mathcal{O}(n^{1/2 -\ve}) \; \; \; \mbox{as $n \to \infty$ almost-surely in } \Omega .
\]
Then we say that the process $(g \circ T^j)_{j \in \mathbb{N}}$ satisfies an almost-sure invariance principle.

\begin{theorem}
\label{thm:limit}
Let $T$ be a Baker's map as above.
If\,\footnote{If one wishes to prove limit theorems for less regular $g \in C^\gamma(M)$ for $\gamma < 1$, then one merely 
needs to reduce the regularity of the test functions used in definition of the norms.  Choosing $\alpha < \gamma < 1$
and requiring $\vf \in C^\alpha(W)$ in $\| \cdot \|_s$ and $\vf \in C^\gamma(W)$ in $\| \cdot \|_u$ and $| \cdot |_w$ would
suffice, for example.  This is done in \cite{dz1} and more generally in \cite{dz3} for observables with singularities. }
  $g \in C^1(M)$, then
\begin{enumerate}
  \item[(a)]  $g$ satisfies a large deviation estimate with the same rate function $I$ for
  all (not necessarily invariant) probability measures $\nu \in \cB$.
\end{enumerate}
Assume that $\bmu(g)=0$, let $\nu \in \cB$ be a probability measure and distribute
$(g \circ T^j)_{j \in \mathbb{N}}$ according to $\nu$.  Then,
\begin{enumerate}
  \item[(b)]  $(g, T, \nu)$ satisfies the Central Limit Theorem;
  \item[(c)]  the process $(g \circ T^j)_{j \in \mathbb{N}}$ satisfies an almost-sure invariance principle.
\end{enumerate}
\end{theorem}

\noindent
The proof of this theorem is contained in Section~\ref{limit proofs}.

Although we have stated Theorem~\ref{thm:limit} in the present context of the Baker's map $T_{\kappa, \lambda}$,
it holds under quite general assumptions on the Banach spaces and transfer operator $\Lp$, as can be seen from the proof
in Section~\ref{limit proofs}.  The essential ingredients we use are:
\begin{itemize}
  \item[(i)]  $\Lp$ has a spectral gap on $\cB$;
  \item[(ii)] the perturbed operator $\Lp_{zg} f (\psi) = f(e^{zg} \cdot \psi \circ T)$ is a continuous perturbation of $\Lp$ in the
  $\cB$-norm for $z \in \mathbb{C}$ close to 0.
\end{itemize}
Indeed, these properties are stronger than what is really needed, especially for items (b) and (c) of the theorem,
where a spectral gap is not needed.
See \cite{gouezel} for a more general approach with weaker assumptions.  Property (ii) can be verified by proving
a type of bounded multiplier lemma such as Lemma~\ref{lem:mult}, which is then applied as in the proof of
Lemma~\ref{lem:analytic}.  Thus the observable $g$ need not be smooth as long as it has this type of bounded multiplier
property in $\cB$.

Once (i) and (ii) are known, part (a) of the theorem follows taking $z$ real, while parts (b) and (c) follow taking $z$ purely imaginary.



\subsection{The singular limit $\lambda \to 0$}
\label{sec:sing}

We establish a formal connection between a family of Baker's maps $T_{\kappa, \lambda}$ and the 
singular transformation obtained in the limit $\lambda \to 0$ by showing that the spectra and spectral projectors of the
associated transfer operators vary continuously in the $\cB_w$ norm.  
This result may be of interest to the study of so-called slap maps derived from polygonal billiards,
which consider similar degenerate limits of the associated piecewise hyperbolic maps \cite{magno1, magno2}.

To make this precise, fix $\kappa \ge 2$ and choose $\kappa$ distinct unstable leaves $U_i \in \cW^u$, $i = 1, \ldots, \kappa$.
Define a family of Baker's transformations $(T_{\kappa, \lambda})_{\lambda}$
such that the $\kappa$ vertical rectangles $R_i$ defined in Section~\ref{bake} on which 
$T_{\kappa, \lambda}$ is smooth
is the same for all $\lambda$, and for each $\lambda$, $T_{\kappa, \lambda}(R_i) \supset U_i$.
With this definition, the set $T_{\kappa, \lambda}(M)$ converges in the Hausdorff metric to $\cup_{i=1}^\kappa U_i$
as $\lambda \to 0$.

Since $\kappa$ is fixed, we will drop the subscript and refer to our family of maps as $(T_\lambda)_{\lambda \in [0, \lambda_0]}$
for some $\lambda_0 \le 1/\kappa$.  When $\lambda = 0$, $T_0$ is the singular transformation which maps $M$ onto 
$\oM_0 := \cup_{i=1}^\kappa U_i$.

Starting from Lebesgue measure $m$, it follows that $(T_0)_*m$ is singular with respect to $m$
and supported on $\oM_0$.  In order to study the evolution of measures,
we recall the notion of standard pairs, which were briefly mentioned
in Remark~\ref{rem:pairs}.

For our purposes, a {\em standard pair} is a measure supported on a single $U \in \cW^u$ together with a smooth density.
Let $\delta_U$ denote the uniform (Lebesgue) measure restricted to $U \in \cW^u$.  
For $f \in C^1(\cW^u)$, we denote
by $f\delta_U$ the standard pair supported on $U \in \cW^u$ with density $f$, i.e.
\[
f\delta_U(\psi) = \int_U \psi \, f \, dm_U, \qquad \mbox{for all } \psi \in C^0(M) ,
\]
where $m_U$ is Lebesgue measure on $U$.
The first key fact allowing us to work with such measures is the following.

\begin{lemma}
\label{lem:stand}
Standard pairs belong to $\cB$.
\end{lemma}

\noindent
This lemma is proved in Section~\ref{sec:pairs}.

\subsubsection{The transfer operator $\Lp_0$}

Despite the singularity of $T_0$, it still holds as in \eqref{eq:dist def}, that
\[
\Lp_0^n f(\vf) = f(\vf \circ T_0^n), \qquad \mbox{for all $f \in (C^\alpha(\cW^s))^*$ and $\vf \in C^\alpha(\cW^s)$.}
\]

Yet we will find it convenient that $\Lp_0 f$ can be
viewed in the sense of distributions as a limit of $\Lp_\lambda f$ as $\lambda \to 0$, where $\Lp_\lambda$ denotes the transfer operator
associated to $T_\lambda$.  

\begin{lemma}
\label{lem:sing L}
For $f \in C^1(\cW^u)$, we have,
\begin{equation}
\label{eq:L0 def}
\Lp_0 f = \sum_{i=1}^\kappa \kappa^{-1} \of \circ T_0^{-1} \, \delta_{U_i},
\end{equation}
where $\of(x) = \int_{W_x} f \, dm_{W_x}$, and $W_x$ is the element of $\cW^s$ containing $x$. 
\end{lemma}

\begin{proof}
For $W \in \cW^s$ and $i = 1, \ldots, \kappa$, we consider $W \cap T_\lambda(R_i)$.  Recall that $T_\lambda(R_i)$ is a 
horizontal strip of width $\lambda$
containing $U_i$.  For $\vf \in C^1(W)$, we have
\[
\int_{W \cap T_\lambda(R_i)} \Lp_\lambda f \, \vf \, dm_W = \kappa^{-1} \int_{W'} f \, \vf \circ T_\lambda \, dm_{W'},
\]
where $W' = T_\lambda^{-1}(W \cap T_{\lambda}(R_i))$.  Note that by assumption on the family $T_\lambda$, $W'$ is independent
of $\lambda$.  Taking the limit as $\lambda \to 0$ yields,
\[
\lim_{\lambda \to 0} \int_{W \cap T_\lambda(R_i)} \Lp_\lambda f \, \vf \, dm_W 
= \kappa^{-1} \left( \int_{W'} f \, dm_{W'} \right) \vf(W \cap U_i),
\]
where $W' = T_0^{-1}(W \cap U_i)$.
Summing over $i$ proves \eqref{eq:L0 def}.
\end{proof}

\begin{remark}
The lemma shows that $\Lp_0f$ is the sum of standard pairs. Perhaps more importantly, the proof of the 
lemma makes clear how $\Lp_0f$ (and
standard pairs in general) should
be viewed 
as a distribution on each stable leaf $W \in \cW^s$:  On $W$, $\Lp_0f$ acts as a sum of point 
masses with weights $\kappa^{-1} \of \circ T_0^{-1}$.  
\end{remark}

Next, iterating \eqref{eq:L0 def} for $n \ge 1$, we obtain,
\[
\Lp_0^n f = \sum_{i=1}^\kappa \kappa^{-1} \bar g \delta_{U_i}, \quad \mbox{where} \quad
\bar g(x) = \kappa^{-n+1} \sum_{T_0^nW = x} \of_W , \quad \mbox{for $x \in U_i$,}
\]
and the cardinality of $\{W: T_0^nW = x \}$ is $\kappa^{n-1}$ for each $x \in U_i$, $i = 1, \ldots \kappa$.

In Section~\ref{singular}, we extend the characterization of $\Lp_0 f$ (and the meaning of $\of$) to all $f \in \cB$.
With these preliminaries established, we can state the following result, which is proved in Section~\ref{quasi L0}

\begin{theorem}
\label{thm:L0}
$\Lp_0$ is bounded as an operator on $\cB$ and $\cB_w$.  It is quasi-compact as an operator
on $\cB$:  its spectral radius is 1, while its essential spectral radius is at most
$\kappa^{-\beta}$.  

Indeed, $\Lp_0$ has a spectral gap on $\cB$ and its unique invariant
probability measure in $\cB$ is given by $\mu_0 = \kappa^{-1} \sum_{i=1}^\kappa  \delta_{U_i}$.  
As a consequence, there exists $C > 0$ and $\sigma < 1$ such that for any
$f \in \cB$ with $f(\bo) = 1$, we have
\[
\| \Lp_0^n f - \mu_0 \|_{\cB} \le C \sigma^n \|f \|_{\cB} , \qquad \mbox{for all $n \ge 0$.}
\]
\end{theorem}


\subsubsection{Perturbation theory for $\Lp_0$}

Unfortunately, it does not follow from the set-up above 
that $\Lp_\lambda \to \Lp_0$ in $\cB$.  However, using an
approach introduced by Keller and Liverani \cite{KL}, we adopt a more flexible approach and view the operators
as acting from $\cB$ to $\cB_w$.  From this perspective, $\Lp_0$ can be viewed as a continuous limit of $\Lp_\lambda$,
and the spectral data can be shown to vary H\"older continuously in $\lambda$.

It is worth remarking on the wide use of this approach, which greatly broadens the classes of perturbations of hyperbolic
dynamical systems able to be studied via the transfer operator.  In particular, it has been used extensively in the study
of open systems \cite{liv maume, kell liv 2, ferg poll, DT}, random perturbations \cite{bahsoun}, hitting times and extreme events  
\cite{keller}, computations of linear response \cite{galatolo linear}, and perturbations of dispersing billiards \cite{dz2}.

To implement this idea, define the following norm for an operator $\Lp: \cB \to \cB_w$,
\[
||| \Lp ||| = \sup\{ |\Lp f |_w : \| f \|_{\cB} \le 1 \} .
\] 
The principal innovation of \cite{KL} is that instead of requiring that a family of operators converge in the strong $\| \cdot \|_{\cB}$ norm, we require
only that they converge in the weaker $||| \cdot |||$ norm.  Provided that they obey a uniform set of Lasota-Yorke inequalities of the form
given by Proposition~\ref{prop:ly} and Lemma~\ref{lem:L0 bounded}, this convergence is sufficient to conclude the 
convergence of their spectra and spectral projectors outside the disk of radius $\kappa^{-\beta}$.

\begin{theorem}
\label{thm:conv}
Let $(\Lp_\lambda)_{\lambda \in [0, \lambda_0]}$ be the family of transfer operators associated
with $(T_\lambda)_{\lambda \in [0, \lambda_0]}$ as defined above.  We have
\begin{equation}
\label{eq:close}
||| \Lp_0 - \Lp_\lambda ||| \le \lambda^{1-\alpha} , \quad \forall \, \lambda \in [0, \lambda_0] .
\end{equation}
Thus the spectra and spectral projectors of $\Lp_0$ and $\Lp_\lambda$ vary H\"older continuously
in $\lambda$ outside the disk of
radius $\kappa^{-\beta}$.  In particular, if $\mu_\lambda$ denotes the unique invariant probability 
measure in $\cB$ for $T_\lambda$, then
$| \mu_\lambda - \mu_0 |_w \to 0$ as $\lambda \to 0$, where 
$\mu_0 = \kappa^{-1} \sum_{i=1}^\kappa \delta_{U_i}$.
\end{theorem}

\noindent
Theorem~\ref{thm:conv} is proved in Section~\ref{perturb}.
 A remarkable feature of the present framework is that the same Banach space can be used to characterize the 
 spectra of the transfer operators for the hyperbolic maps $T_\lambda$ as for the expanding map $T_0$. 
 For another variation on this type of connection, see \cite{demers}. 


\section{Properties of the Banach Spaces}
\label{props}

We begin by proving Lemma~\ref{lem:embeddings} and then establish some preliminary facts
about our spaces.

\begin{proof}[Proof of Lemma~\ref{lem:embeddings}]
The first and third embeddings are trivial:  We have the easy inequalities $| f |_{C^1(M)} \ge | f |_{C^{\beta'}(\cW^u)}$ and 
$\| f \|_s \ge | f |_w$, proving the continuity of the inclusions, while the injectivity of $C^1(M)$ in $C^{\beta'}(\cW^u)$ is obvious,
and that of $\cB$ in $\cB_w$ follows from the fact that we have defined $C^\alpha(W)$ so that $C^1(W)$ is dense in 
$C^\alpha(W)$ for all $W \in \cW^s$.
The fact that $\cB$ is relatively compact in $\cB_w$ is proved in Lemma~\ref{lem:compact}.

For the second inclusion, it is not enough that $| f |_{C^{\beta'}(\cW^u)} \ge \| f \|_{\cB}$.  Since we have defined $\cB$ to be the
completion of $C^1(\cW^u)$ in the $\| \cdot \|_{\cB}$ norm, we must approximate $f \in C^{\beta'}(\cW^u)$ by 
functions $g_\ve \in C^1(\cW^u)$.  

Let $\beta' \in (\beta, 1]$.  Let $\rho \ge 0$ be a $C^\infty$ bump function on $\mathbb{R}$, supported on $[-1,1]$, 
such that $\int_{-1}^1 \rho \, dt = 1$.  For $\ve > 0$, define $\rho_\ve(t) = \ve^{-1} \rho(t/\ve)$.
For $f \in C^{\beta'}(\cW^u)$, and $U \in \cW^u$, let $x = (s, t_U)$, $s \in [0,1]$, denote coordinates on 
$U$ (where $t_U \in [0,1]$ is fixed for all $x \in U$).  Define
the convolution,
\begin{equation}
\label{eq:convol}
g_\ve(x) = g_\ve(u, t_U) = \int_U f(y) \, \rho_\ve(x-y) \, dm_U(y) = \int_0^1 f(s, t_U) \, \rho_\ve(u-s) \, ds .
\end{equation}
It is an elementary estimate that $g_\ve \in C^1(\cW^u)$ with $|g_\ve|_{C^1(\cW^u)} \le C \ve^{\beta'-1} |f|_{C^{\beta'}(\cW^u)}$,
and $|g_\ve|_{C^{\beta'}(\cW^u)} \le |f|_{C^{\beta'}(\cW^u)}$, for some $C>0$ depending on $\rho$.  

We claim that $|g_\ve - f|_{C^\beta(\cW^u)} \le 4 |f|_{C^{\beta'}(\cW^u)} \ve^{\beta'-\beta}$.  
To see this, fix $U \in \cW^u$ and the vertical coordinate
$t_U$ corresponding to $U$.  Then for any $u \in U$, using the notation in \eqref{eq:convol}, we have,
\[
|g_\ve(u) - f(u)| \le \int_0^1 |f(u) - f(s)| \rho_\ve(u-s) ds \le H^\beta_U(f) \int_{u-\ve}^{u+\ve} |u-s|^{\beta'} \rho_\ve(u-s) ds \le \ve^{\beta'} H^{\beta'}_U(f).
\]
Then for $u, v \in U$, we have on the one hand,
\[
|(g_\ve-f)(u) - (g_\ve -f)(v)| \le 2 \ve^{\beta'} H^{\beta'}_U(f),
\]
while on the other hand,
\[
|(g_\ve-f)(u) - (g_\ve -f)(v)| \le |g_\ve(u) - g_\ve(v)| + |f(u) - f(v)| \le 2 |u-v|^{\beta'} H^{\beta'}_U(f)  .
\]
Interpolating, we conclude that\footnote{Here we use that for any $A, B > 0$ and $q \in (0,1)$,
$\min \{ A, B \} \le A^{1-q} B^q$.  For us, $q = \beta/\beta'$.}
\begin{equation}
\label{eq:interpolate}
|(g_\ve-f)(u) - (g_\ve -f)(v)| \le 2H^{\beta'}_U(f) \min \{\ve^{\beta'}  |u-v|^{-\beta}, |u-v|^{\beta'-\beta} \} \le 2 H^{\beta'}_U(f) \ve^{\beta'-\beta}.
\end{equation}
Putting these estimates together, we have $|g_\ve - f|_{C^\beta(\cW^u)} \le 4 |f|_{C^{\beta'}(\cW^u)} \ve^{\beta' - \beta}$, as claimed.

It is clear that, $\| g_\ve \|_s \le |g_\ve|_{C^0(\cW^u)} \le |f|_{C^0(\cW^u)}$. 

Next, for
any $W_1, W_2 \in \cW^s$ and $\vf_i \in C^1(W_i)$ with $|\vf_i|_{C^1(W_i)} \le 1$ and $d_0(\vf_1, \vf_2) =0$, we have
\[
\int_{W_1} g_\ve \, \vf_1 \, dm_{W_1} - \int_{W_2} g_\ve \, \vf_2 \, dm_{W_2}
= \int_0^1 (g_\ve(s_{W_1}, t) - g_\ve(s_{W_2}, t)) \vf_1(s_{W_1},t) \, dt \\
\le |s_{W_1} - s_{W_2}|^\beta |g_\ve|_{C^\beta(\cW^u)} ,
\]
where $(s_{W_i}, t)$, $t \in [0,1]$, are the coordinates in $W_i$ according to \eqref{eq:coords}.
Thus $\| g_\ve \|_u \le C |f|_{C^\beta(\cW^u)}$, so that the family $(g_\ve)_{\ve > 0}$ is bounded in $\cB$.  

The above estimates also show that $\| g \|_{\cB} \le |g|_{C^\beta(\cW^u)}$ for any $g \in C^\beta(\cW^u)$.
Thus 
\[
\| g_\ve - f \|_{\cB} \le |g_\ve - f|_{C^\beta(\cW^u)} \le 4\ve^{\beta'-\beta} |f|_{C^{\beta'}(\cW^u)} .
\]  
We conclude that
the sequence $g_\ve$ converges to $f$ in the $\| \cdot \|_{\cB}$ norm, and thus $f \in \cB$.  

We have shown $C^{\beta'}(\cW^u) \subset \cB$, if $\beta' > \beta$, and the continuity of this inclusion
follows from $| f |_{C^{\beta'}(\cW^u)} \ge \| f \|_{\cB}$.  Injectivity may fail if, for example, two functions in $C^{\beta'}(\cW^u)$
differ on a zero measure set of unstable leaves.

Finally, we turn to the fourth inclusion.  Let $\vf \in C^1(\cW^s)$ and $f \in C^1(\cW^u)$.  We index the elements of $\cW^s$ by
$W_s$, $s \in [0,1]$.  Then, recalling the identification of $f$ with the measure $f dm$,
\[
|f(\vf)| =  \left| \int_M f \, \vf \, dm \right| = \left| \int_0^1 \int_{W_s} f \, \vf \, dm_{W_s} ds \right|
\le \int_0^1 | f |_w |\vf|_{C^1(W_s)} \, ds \le | f |_w |\vf|_{C^1(\cW^s)} .
\]
By density, this bound extends to all $f \in \cB_w$, which proves the continuity of the embedding.  

To prove the injectivity, consider $f \in C^1(\cW^u)$ and $W \in \cW^s$.    Define,
\begin{equation}
\label{eq:distribution}
\langle D^1_W(f) , \vf \rangle := \int_W f \, \vf \, dm_W, \qquad \vf \in C^1(\cW^s).
\end{equation}
Since $|\langle D^1_W(f) , \vf \rangle| \le |f|_w |\vf|_{C^1(\cW^s)}$, $D^1_W(f)$ is an element of $(C^1(\cW^s))^*$.
Moreover, since the map $f \mapsto D^1_W(f)$ is continuous in the $| \cdot |_w$ norm, it can be extended to all of $\cB_w$.
Thus each element of $f \in \cB_w$ defines a family of distributions of order 1 on the leaves
$W \in \cW^s$.

Now suppose $f \in \cB_w$ with $|f|_w \neq 0$.  Then there exists $W_0 \in \cW^s$ and $\vf \in C^1(W_0)$ such that
$\langle D^1_{W_0}(f), \vf \rangle = \delta > 0$.  We may extend $\vf$ to a function $\bvf$ on $M$ by making $\bvf$ constant
on unstable leaves.  Then $\bvf \in C^1(\cW^s)$.  For $g \in C^1(\cW^u)$, since $\| g \|_u \le |g|_{C^1(\cW^u)}$, the map
$W \mapsto \langle D^1_W(g), \bvf \rangle$ is continuous in $W$.  By density, the map is continuous for all $g \in \cB_w$.

Thus there exists $\ve > 0$ such that for all $W \in \cW^s$ such that $d(W, W_0) < \ve$, we have 
$\langle D^1_W(f), \bvf \rangle \ge \delta/2$.  Now define $\bvf_\ve = \bvf \cdot 1_{N_\ve(W_0)}$, where $N_\ve(\cdot)$ denotes the
$\ve$ neighborhood of a set in $M$.  It still holds that $\bvf_\ve \in C^1(\cW^s)$.  Then indexing the stable leaves
in $N_\ve(W_0)$ by $W_s$, $s \in (-\ve, \ve)$, we estimate
\[
f(\bvf_\ve) = \int_{-\ve}^\ve \langle D^1_{W_s}(f), \bvf_\ve \rangle \, ds \ge \ve \delta > 0.
\]
Thus $f \neq 0$ as an element of $(C^1(\cW^s))^*$ and the inclusion $\cB_w \hookrightarrow (C^1(\cW^s))^*$ is injective.
\end{proof}

The following multiplier lemma is convenient for establishing the spectral decomposition and limit theorems for $T$.

\begin{lemma}
\label{lem:mult}
If $g \in C^1(M)$ and $f \in \cB$, then $gf \in \cB$ and $\| g f \|_{\cB} \le 3 |g|_{C^1(M)} \| f \|_{\cB}$.
\end{lemma}

\begin{proof}
By density, it suffices to prove the statement for $f \in C^1(\cW^u)$ and $g \in C^1(M)$.  
It is clear that $fg \in C^1(\cW^u)$.
We proceed to estimate the strong norm of the product.

Let $W \in \cW^s$ and $\vf \in C^\alpha(W)$ such that $|\vf|_{C^\alpha(W)} \le 1$.  
Then
\[
\int_W f \, g \, \vf \, dm_W \le \| f \|_s |g\vf|_{C^\alpha(W)} \le \| f \|_s |g|_{C^\alpha(W)} |\vf_{C^\alpha(W)} \le \| f \|_s |g|_{C^1(M)} .
\]
Next, choose $W_1, W_2 \in \cW^s$, and $\vf_i \in C^1(W_i)$ such that $|\vf_i|_{C^1(W)} \le 1$, $i=1,2$, and 
$d_0(\vf_1, \vf_2) = 0$.  Let $\Theta : W_2 \to W_1$ be the holonomy map sliding along unstable leaves and define
$\tilde g = g \circ \Theta$ on $W_2$.  Note that $|\tilde g|_{C^1(W_2)} \le |g|_{C^1(M)}$.  Then
\[
\int_{W_1} f \, g\, \vf_1 \, dm_{W_1} - \int_{W_2} f \, g \, \vf_2 \, dm_{W_2}
= \int_{W_1} f \, g\, \vf_1 \, dm_{W_1} - \int_{W_2} f \, \tilde g \, \vf_2 \, dm_{W_2} + \int_{W_2} f \, (\tilde g - g) \vf_2 \, dm_{W_2}.
\]
Since $d_0(g \vf_1, \tilde g \vf_2) = 0$, we may estimate the difference in integrals on the right hand side by 
$\| f \|_u |g|_{C^1(M)} d(W_1, W_2)^\beta$.  For the third integral on the right hand side, we estimate using the strong stable norm,
\begin{equation}
\label{eq:stable error}
\int_{W_2} f \, (\tilde g - g) \vf_2 \, dm_{W_2} \le \| f \|_s |\tilde g - g|_{C^\alpha(W_2)} |\vf_2|_{C^\alpha(W_2)} . 
\end{equation}
Since $g \in C^1(M)$, we have for $x \in W_2$,
\[
|\tilde g(x) - g(x)| \le |g|_{C^1(M)} |x - \Theta(x)| \le |g|_{C^1(M)} d(W_1, W_2) .
\]
While for $x, y \in W_2$, we have on the one hand
\[
d(x,y)^{-\alpha}|(\tilde g - g)(x) - (\tilde g - g)(y)| \le 2 |g|_{C^1(M)} d(x,y)^{-\alpha} d(W_1, W_2),
\]
while on the other  hand, using the fact that $\tilde g$ and $g$ are separately $C^1$,
\[
d(x,y)^{-\alpha}|(\tilde g - g)(x) - (\tilde g - g)(y)| \le 2 |g|_{C^1(M)} d(x,y)^{1-\alpha}.
\]
Interpolating as in \eqref{eq:interpolate}, we conclude,
\[
|\tilde g - g|_{C^\alpha(M)} \le 2|g|_{C^1(M)} d(W_1, W_2)^{1-\alpha} .
\]
Using this estimate together with \eqref{eq:stable error}, we have
\[
\left| \int_{W_1} f \, g\, \vf_1 \, dm_{W_1} - \int_{W_2} f \, g \, \vf_2 \, dm_{W_2} \right|
\le \big( \| f \|_u  d(W_1, W_2)^\beta
+ 2 \| f \|_s  d(W_1, W_2)^{1-\alpha} \big)  |g|_{C^1(M)}  .
\]
Since $1-\alpha \ge \beta$, we divide through by $d(W_1, W_2)^\beta$ and take the suprema over $W_1, W_2 \in \cW^s$
and $\vf_i \in C^1(W_i)$ with $|\vf_i|_{C^1(W_i)} \le 1$ to obtain the required estimate.
\end{proof}

\begin{lemma}
\label{lem:compact}
The unit ball of $\cB$ is compactly embedded in $\cB_w$.
\end{lemma}

\begin{proof}
The proof proceeds in two steps.  Fix $\ve > 0$.  First, for each $W \in \cW^s$, the unit ball of $C^1(W)$
is compactly embedded in $C^\alpha(W)$.  Let $C_1^1(W) = \{ \vf \in C^1(W) : |\vf|_{C^1(W)} \le 1 \}$
be the unit ball in $C^1(W)$.
By compactness, there exists a finite collection of functions $\vf_i$, $i = 1, \ldots N(\ve)$, such that
$\vf_i \in C_1^1(W)$ and the set $\{ \vf_i \}_{i=1}^N$ forms an $\ve$-cover of $C_1^1(W)$ in the 
$| \cdot |_{C^\alpha(W)}$ norm.  Indeed, we can extend $\vf_i$ to $M$ by defining the extension 
$\bvf_i$ to be constant
on unstable leaves.  Then $\{ \bvf_i \}_{i=1}^N$ forms an $\ve$-cover of $C_1^1(W')$ in the
$| \cdot |_{C^\alpha(W')}$ norm for any $W' \in \cW^s$. 

Next, choose a finite set $\{ W_j \}_{j=1}^L \subset \cW^s$ which forms an $\ve$-cover of 
$\cW^s$ in the metric $d(\cdot, \cdot)$ between stable leaves.

Now for $f \in C^1(\cW^u)$, $W \in \cW^s$ and $\vf \in C_1^1(W)$,
choose $W_j$ and $\bvf_i$ such that $d(W_j, W) \le \ve$ and $|\bvf_i - \vf|_{C^\alpha(W)} \le \ve$.
Then,
\[
\begin{split}
\left| \int_W f \, \vf \, dm_W \right.  & - \left. \int_{W_j} f \, \bvf_i \, dm_{W_j} \right|
 \le \left| \int_W f \, (\vf - \bvf_i) \, dm_W\right| + \left| \int_W f \, \bvf_i \, dm_W - \int_{W_j} f \, \bvf_i \, dm_{W_j} \right| \\
& \le \| f\|_s |\vf - \bvf_i |_{C^\alpha(W)} + |s_W - s_{W_j}|^\beta \| f \|_u \le \ve \|f \|_s + \ve^\beta \| f \|_u
\le \ve^\beta \| f \|_{\cB} . 
\end{split}
\]
Taking the supremum over $W$ and $\vf$ yields that $|f|_w$ can be approximated by the
finite set of linear functionals $\ell_{i,j}(f) = \int_{W_j} f \, \bvf_i \, dm_{W_j}$, $1 \le i \le N$, $1 \le j \le L$.  This implies the
required compactness by the following argument. 

Let $\cB_1$ be the unit ball in $\cB$.  If $f\in \cB_1$, then $\{ \ell_{i,j}(f) \}_{i,j}$ can be identified with a vector $(\ell_{i,j}(f))_{i,j}$ in the unit ball
of $\mathbb{R}^{N+L}$.  Since this is a compact set, there exists a finite set $\{ g_n \}_{n=1}^K \subset \cB_1$
whose set of associated vectors $\{ (\ell_{i,j}(g_n))_{i,j} \}_{n}$ forms an $\ve$-cover of the unit ball
in $\mathbb{R}^{N+L}$.  By the triangle inequality, any $f \in \cB_1$ can be approximated 
up to $\ve^\beta$ in the weak norm by an element of the set $\{ g_n \}_{n=1}^K$.
\end{proof}


\section{Spectral Properties of the Transfer Operator}

In this section, we prove Theorem~\ref{thm:spec} in a series of steps, first proving quasi-compactness
of the transfer operator acting on $\cB$, then a characterization of the peripheral spectrum, and finally the
existence of a spectral gap.  We begin by establishing the continuity of $\Lp$ on our spaces
$\cB$ and $\cB_w$.

\begin{lemma}
\label{lem:cont}
The transfer operator $\Lp$ is a bounded, linear operator on both $\cB$ and $\cB_w$.
\end{lemma}

\begin{proof}
The proof is straightforward given that $\Lp$ preserves\footnote{In many systems with discontinuities,
it is not so simple to find a space like $C^1(\cW^u)$ that is preserved by $\Lp$.  In such cases,
this proof must proceed by approximation, for example approximation of $\Lp f$ by smooth functions
in the $\cB$ norm.  See, for example, \cite[Lemmas~3.7 and 3.8]{dz1}.} the space $C^1(\cW^u)$.  More precisely,
if $f \in C^1(\cW^u)$, then it is clear from \eqref{eq:trans def} that $\Lp f \in C^1(\cW^u)$ since
for any $U \in \cW^u$, $T^{-1}U \subset U'$ for some $U' \in \cW^u$.  Indeed,
$|\Lp f|_{C^0(U)} \le |f|_{C^0(T^{-1}U)} (\kappa \lambda)^{-1}$ and 
$H_U^1(\Lp f) \le H^1_U(f \circ T^{-1}) (\kappa \lambda)^{-1} \le H^1_{T^{-1}U}(f) \kappa^{-2} \lambda^{-1}$.
Taking the supremum over $U \in \cW^u$ yields 
$|\Lp f |_{C^1(\cW^u)} \le |f|_{C^1(\cW^u)} (\kappa \lambda)^{-1}$.

Finally, Proposition~\ref{prop:ly} proves that for $f \in \cB$, $\| \Lp f \|_{\cB} \le C \| f \|_{\cB}$ for some
$C>0$ independent of $f$.  Thus if $g \in \cB$ and
$(f_n)_{n \in \mathbb{N}} \subset C^1(\cW^u)$ is a sequence converging to $g$ in the  
$\| \cdot \|_{\cB}$ norm, then
$(\Lp f_n)_{n \in \mathbb{N}} \subset C^1(\cW^u)$ is a sequence converging to $\Lp g$ in the $\cB$ norm,
since $\| \Lp f_n - \Lp g \|_{\cB} \le C \| f_n - g\|_{\cB}$.  Since $\cB$ is the completion of $C^1(\cW^u)$,
we conclude 
$\Lp g \in \cB$.

A similar argument holds for $\cB_w$.
\end{proof}

\subsection{Quasicompactness of $\Lp$}

\begin{proposition}
\label{prop:ly}
For any $n \ge 0$ and $f \in \cB$,
\begin{eqnarray}
  \| \Lp^n f \|_s & \le & \lambda^{\alpha n} \| f \|_s + |f|_w, \label{eq:strong stable}  \\ 
  \| \Lp^n f \|_u & \le & \kappa^{-\beta n} \| f \|_u,  \label{eq:strong unstable}   \\
  | \Lp^n f |_w  & \le & |f|_w.   \label{eq:weak}
\end{eqnarray}
\end{proposition}

\begin{proof}
By density of $C^1(\cW^u)$ in both $\cB$ and $\cB_w$, it suffices to prove the bounds for
$f \in C^1(\cW^u)$.  

For $W \in \cW^s$ and $n \ge 1$, we note that $T^{-n}W = \cup_i W_i$, where each $W_i \in \cW^s$.  We call this
set the $n$th generation of $W$ and denote it by $\cG_n(W)$.
It follows from the definition of $T$ that the cardinality of $\cG_n(W)$ is $\kappa^n$.

We begin by proving \eqref{eq:weak}.

\medskip
\noindent
{\em Weak Norm Bound.}  Let $f \in C^1(\cW^u)$, $W \in \cW^s$, and suppose $\vf \in C^1(W)$ with $|\vf|_{C^1(W)} \le 1$.
For $n \ge 1$, we estimate,
\begin{equation}
\label{eq:first sum}
\begin{split}
\int_W \Lp^n f \, \vf \, dm_W & = \sum_{W_i \in \cG_n(W)} \int_{W_i} f \, \vf \circ T^n \, (\lambda\kappa)^{-n} 
\, J^s_{W_i}T^n \,dm_{W_i} \\
& = \sum_{W_i \in \cG_n(W)} \kappa^{-n} \int_{W_i} f \, \vf \circ T^n \, dm_{W_i},
\end{split}
\end{equation}
where $J^s_{W_i}T^n = \lambda^n$ represents the stable Jacobian of $T^n$ along $W_i$.  
Now for $x, y \in W_i$,
\[
|\vf \circ T^n(x) - \vf \circ T^n(y)| \le H^1_W(\vf) |T^nx - T^ny| \le \lambda^n |x-y|,
\]
so that $|\vf \circ T^n|_{C^1(W_i)} \le |\vf|_{C^1(W)} \le 1$.  Thus applying the definition of the weak norm to each term of
the sum in \eqref{eq:first sum}, we have,
\[
\int_W \Lp^n f \, \vf \, dm_W \le \sum_{W_i \in \cG_n(W)} |f|_w \kappa^{-n} = |f|_w.
\]
Taking the suprema over $\vf \in C^1(W)$ with $|\vf|_{C^1(W)} \le 1$ and $W \in \cW^s$ yields \eqref{eq:weak}.

\medskip
\noindent
{\em Strong Stable Norm Bound.}
As before, let $f \in C^1(\cW^u)$, $W \in \cW^s$, and suppose $\vf \in C^\alpha(W)$ with $|\vf|_{C^\alpha(W)} \le 1$.
For $n \ge 1$, on each $W_i \in \cG_n(W)$, define $\bvf_i = \int_{W_i} \vf \circ T^n \, dm_{W_i}$.  Then following
\eqref{eq:first sum} and \eqref{eq:contract ly}, we write,
\begin{equation}
\label{eq:stable split}
\begin{split}
\int_W \Lp^n f \, \vf \, dm_W & = \sum_{W_i \in \cG_n(W)} \int_{W_i} f \, (\vf \circ T^n - \bvf_i) \, \kappa^{-n} \,dm_{W_i}
+ \int_{W_i} f \, \bvf_i \, \kappa^{-n} \, dm_{W_i} \\
& \le \kappa^{-n}  \sum_{W_i \in \cG_n(W)} \| f \|_s |\vf \circ T^n - \bvf_i|_{C^\alpha(W_i)} + |f|_w |\bvf_i|_{C^1(W_i)}  \, .
\end{split}
\end{equation}
For the first term on the right side of \eqref{eq:stable split}, we use \eqref{eq:avg subtract} to estimate
$|\vf \circ T^n - \bvf_i|_{C^\alpha(W_i)} \le \lambda^{\alpha n} |\vf|_{C^\alpha(W)}$, while $|\bvf_i|_{C^1(W_i)} \le |\vf|_{C^0(W)}$
since $\bvf_i$ is constant on $W_i$.  Putting these together, we conclude,
\[
\int_W \Lp^n f \, \vf \, dm_W \le \kappa^{-n} \sum_{W_i \in \cG_n(W)} \Big( \lambda^{\alpha n} \| f \|_s + |f|_w \Big) \le \lambda^{\alpha n} \| f \|_s + |f|_w,
\]
and taking the appropriate suprema proves \eqref{eq:strong stable}.

\medskip
\noindent
{\em Strong Unstable Norm Bound.}
Let $f \in C^1(\cW^u)$.  For $W^1, W^2 \in \cW^s$, let $\vf_i \in C^1(W^i)$, $i = 1,2$, such that $| \vf_i |_{C^1(W)} \le 1$ and
$d_0(\vf_1, \vf_2) = 0$.  This assumption requires that $\vf_1(x) = \vf_2(y)$ whenever $x$ and $y$ lie on the same horizontal line. 

Note that due to the product structure of the map, the elements of $\cG_n(W^1)$ and $\cG_n(W^2)$ are in one-to-one correspondence:
for each $W_i^1 \in \cG_n(W^1)$, there is a unique element $W_i^2 \in \cG_n(W^2)$ lying in the maximal vertical rectangle 
containing $W_i^1$ on which $T^n$ is smooth.
With this pairing, we estimate,
\[
\begin{split}
\int_{W^1} \Lp^n f \, \vf_1 \, dm_{W_1} - \int_{W_2} \Lp^n f \, \vf_2 \, dm_{W_2} 
& = \kappa^{-n} \sum_{W_i^1 \in \cG_n(W)} \int_{W_i^1} f \, \vf_1 \circ T^n \, dm_{W^1_i} - \int_{W_i^2} f \, \vf_2 \circ T^n \, dm_{W^2_i} \\
& \le \kappa^{-n} \sum_{W_i^1 \in \cG_n(W)} d(W_i^1, W_i^2)^\beta \| f \|_u ,
\end{split}
\]
where we have applied the strong unstable norm since $d_0(\vf_1 \circ T^n, \vf_2 \circ T^n) = 0$ and 
$|\vf_j \circ T^n|_{C^1(W_i^j)} \le |\vf_j|_{C^1(W)} \le 1$.  Since $d(W_i^1, W_i^2) = \kappa^{-n} d(W^1, W^2)$, we conclude
\[
d(W^1, W^2)^{-\beta} \left| \int_{W^1} \Lp^n f \, \vf_1 \, dm_{W_1} - \int_{W_2} \Lp^n f \, \vf_2 \, dm_{W_2} \right|
\le \kappa^{-\beta n} \| f \|_u,
\]
and taking the appropriate suprema proves \eqref{eq:strong unstable}.
\end{proof}

\begin{proposition}
\label{prop:quasi}
The operator $\Lp: \cB \circlearrowleft$ is quasi-compact:  its spectral radius is 1 and its essential spectral radius is bounded by
$\rho = \max \{ \lambda^\alpha, \kappa^{-\beta} \} < 1$.  The part of the spectrum outside the disk of radius $\rho + \ve$
for any $\ve>0$ is finite dimensional.  Finally, the peripheral spectrum of $\Lp$
is semi-simple, i.e. the eigenvalues of modulus 1 have no Jordan blocks. 
\end{proposition}

\begin{proof}
The proposition is a simple corollary of Proposition~\ref{prop:ly}.  The estimates \eqref{eq:strong stable} and \eqref{eq:strong unstable} imply that for any $f \in \cB$ and $n \ge 1$,
\[
\| \Lp^n f \|_{\cB} = \| \Lp^n f \|_s + \| \Lp^n f \|_u \le \lambda^{\alpha n} \| f \|_s + \kappa^{-\beta n} \| f \|_u + |f|_w
\le \rho^n \| f \|_{\cB} + |f|_w .
\]
This, coupled with the compactness of the unit ball of $\cB$ in $\cB_w$ (Lemma~\ref{lem:compact}), 
implies the stated bound on the essential spectral radius
of $\Lp$ by the standard Hennion-Nussbaum argument \cite{nussbaum, hennion spec}. 

The bounds also imply that the spectral radius of $\Lp$ is at most 1.  Since $\Lp^*m = m$, 1 is in the spectrum of $\Lp^*$,
and therefore in the spectrum of $\Lp$ \cite[Theorem VI.7]{reed}.  
Thus the spectral radius of $\Lp$ is $1 > \rho$, and $\Lp$ is quasi-compact. 

With quasi-compactness established, the fact that the peripheral spectrum is semi-simple
follows immediately from the bound $\| \Lp^n f \|_{\cB} \le \| f \|_{\cB}$, for all $n \in \mathbb{N}, f \in \cB$, which in turn follows from the
proof of Proposition~\ref{prop:ly}.
\end{proof}


\subsection{Characterization of the peripheral spectrum}
\label{periph}

The next step needed to complete the spectral decomposition of the transfer operator is a 
characterization of the peripheral spectrum.  The quasicompactness of $\Lp$ coupled with the
absence of Jordan blocks implies that there exist numbers $\theta_j \in [0,1)$ and operators $\Pi_j$, 
$R : \cB \circlearrowleft$, $j=0, \ldots N$, with $\Pi_j^2 = \Pi_j$, 
$\Pi_j R = R \Pi_j = 0$ and $\Pi_j \Pi_k = \Pi_k \Pi_j = 0$ for $j \neq k$, such that
\begin{equation}
\label{eq:decomp}
\Lp = \sum_{j=0}^N e^{2 \pi i \theta_j} \Pi_j + R .
\end{equation}
Moreover, the spectral radius of $R$ is bounded by some $r<1$. 
We denote by $\bV_j$ the eigenspace associated with $\theta_j$ and set $\theta_0 = 0$ so
that $\bV_0$ denotes the space of invariant distributions in $\cB$ (which is nonempty by
Proposition~\ref{prop:quasi} since 1 is in the spectrum).

Due to \eqref{eq:decomp}, we may characterize the spectral projectors $\Pi_j$ by the following
limit,
\[
\Pi_j = \lim_{n \to \infty} \frac 1n \sum_{k=0}^{n-1} e^{-2\pi i \theta_j k} \Lp^k ,
\]
which converges in the uniform topology of $L(\cB, \cB)$.
The following lemma summarizes the main points.

\begin{lemma}
\label{lem:spec}
  Define $\bmu = \Pi_0 \bo$.
  \begin{itemize}
    \item[(i)]  All the elements of $\bV = \oplus_j \bV_j$ are Radon measures absolutely continuous with respect
    to $\bmu$.
    \item[(ii)]  There exist a finite number of $q_k \in \mathbb{N}$ such that 
    $\cup_{j = 0}^N \{ \theta_j \} = \cup_k \{ \frac{p}{q_k} : 0 < p \le q_k, p \in \mathbb{N} \}$.
    \item[(iii)]  The set of ergodic probability measures absolutely continuous with respect to $\bmu$
    form a basis for $\bV_0$.
  \end{itemize}
\end{lemma}

\begin{proof}
(i)
Choose $j \in [0, N]$ and let $\mu \in \bV_j$.  Since $C^1(\cW^u)$ is dense $\cB$ and $\Pi_j$ is a bounded linear operator, $\Pi_j C^1(\cW^u)$
is dense in $\bV_j$.  Since $\bV_j$ is finite dimensional, we have $\Pi_j C^1(\cW^u) = \bV_j$.

Thus there exists $f \in C^1(\cW^u)$ such that $\Pi_j f = \mu$.  Now for each $\vf \in C^1(\cW^s)$,
\[
|\mu(\vf)| = |\Pi_j f(\vf)| \le \lim_{n \to \infty} \frac 1n \sum_{k=0}^{n-1} | \Lp^k f(\vf)|
= \lim_{n \to \infty} \frac 1n \sum_{k=0}^{n-1} |f(\vf \circ T^k)| \le |f|_\infty |\vf|_\infty.
\]
Since $C^1(\cW^s)$ is dense in $C^0(M)$, $\mu$ can be extended to a bounded linear functional
on $C^0(M)$, i.e., $\mu$ is a Radon measure.  This applies, in particular, to $\bmu$.
Moreover, for any $\vf \ge 0$,
\[
|\mu(\vf)| \le \lim_{n \to \infty} \frac 1n \sum_{k=0}^{n-1} |\Lp^k f(\vf)|
\le \lim_{n \to \infty} \frac 1n \sum_{k=0}^{n-1} |f|_\infty \Lp^k \bo (\vf) \le |f|_\infty \bmu(\vf),
\]
so that $\mu$ is absolutely continuous with respect to $\bmu$, and its Radon-Nikodym derivative $f_\mu$
is in $L^\infty(\bmu)$.  


\smallskip
\noindent
(ii) Appealing to the dual of $\Lp$ as in Section~\ref{ped} is more involved in the present setting
since $\cB$ is not exactly the
dual of $C^1(\cW^s)$.  Instead, we prefer the following more
general argument, which uses item (i) of the lemma (see also \cite[Lemma~5.5]{demers liverani}).  
Fixing again $\theta_j$, we
choose $\mu  \in \bV_j$.  According to (i), $\mu = f_\mu \bmu$ for some $f_\mu \in L^\infty(\bmu)$.
Then for $\vf \in C^1(\cW^s)$, we have
\[
e^{2 \pi i \theta_j} \mu(\vf) = \Lp\mu(\vf) = \mu(\vf \circ T) = \int_M \vf \circ T \, f_\mu \, d\bmu
= \int_M \vf \, f_\mu \circ T^{-1} \, d\bmu.
\]
Since this holds for each $\vf \in C^1(\cW^s)$, we conclude that 
$f_\mu \circ T^{-1} = e^{2 \pi i \theta_j} f_\mu$, $\bmu$-almost everywhere.  Let
$f_{\mu,k} = (f_\mu)^k$.  Then $f_{\mu, k} \in L^\infty(\bmu)$ and
$f_{\mu,k} \circ T = e^{-2 \pi i \theta_j k} f_{\mu,k}$.

We claim that
$\mu_k = f_{\mu, k}\, \bmu \in \cB$.
To see this, fix $\ve > 0$ and choose $g \in C^1(M)$ such that $\bmu(|g - f_{\mu, k}|) \le \ve$.
By Lemma~\ref{lem:mult}, $g\bmu \in \cB$, and for $\vf \in C^1(\cW^s)$,
\[
\begin{split}
\lim_n \frac 1n & \sum_{\ell=0}^{n-1} e^{-2 \pi i \theta_j k \ell} \Lp^\ell (g \bmu)(\vf) - ((f_\mu)^k \bmu)(\vf) \\
& =  \lim_n \frac 1n \sum_{\ell=0}^{n-1} e^{-2 \pi i \theta_j k \ell} \bmu(g \, \vf \circ T^\ell) 
- e^{-2 \pi i \theta_j k} \bmu((f_{\mu,k} \, \vf \circ T^k) \\
& = \lim_n \frac 1n \sum_{\ell=0}^{n-1} e^{-2 \pi i \theta_j k \ell} \bmu((g - f_{\mu,k}) \, \vf \circ T^\ell) 
\le  \lim_n \frac 1n \sum_{\ell=0}^{n-1}  \bmu(|g - f_{\mu,k}| \circ T^{-\ell} ) |\vf|_\infty
\le \ve |\vf|_\infty.
\end{split}
\]
Since $f_\mu \neq 0$, this calculation shows that (a) 
$\lim_n \frac 1n \sum_{\ell=0}^{n-1} e^{-2 \pi i \theta_j k \ell} \Lp^\ell (g \bmu) \neq 0$, so that
$k \theta_j$ is in the spectrum of $\Lp$, and (b) $\mu_k = f_{\mu, k} \, \bmu$ can be approximated by
elements of $\bV$ and so must belong to $\bV$.  Since this is true for each $k$ and the peripheral
spectrum of $\Lp$ is finite, we must have $k \theta_j = 0$ (mod 1) for some $k$, proving the claim.

\smallskip
\noindent
(iii)  Let $\mu \in \bV_0$ and $f \in C^1(\cW^u)$ such that $\Pi_0f = \mu$.  Setting 
$f^+ = \max \{ f, 0 \}$ and $f^- = \max \{ -f, 0 \}$, we have $f^{\pm} \in C^1(\cW^u)$ and 
$f = f^+ - f^-$.  Then defining $\mu^{\pm} = \Pi_0 f^{\pm}$, we have $\mu^{\pm} \ge 0$ and
$\mu  = \mu^+ - \mu^-$,  so that $\bV_0$ is the span of a convex set of probability measures.

Next, suppose that $A \subset M$ is an invariant set with $\bmu(A)>0$.  For each $\ve >0$,
there exists $\vf_\ve \in C^1(M)$ such that $\bmu(|\bo_A - \vf_\ve|) < \ve$, where
$\bo_A$ is the indicator function of $A$.  Thus for each $\psi \in C^0(M)$, we have
\[
\bmu(\vf_\ve \, \psi \circ T^n) = \bmu(\bo_A \, \psi \circ T^n) + \mathcal{O}(\ve|\psi|_\infty)
= \bmu(\bo_A \, \psi) + \mathcal{O}(\ve|\psi|_\infty)
= \bmu(\vf_\ve \, \psi) + \mathcal{O}(\ve|\psi|_\infty),
\]
where we have used the invariance of $A$ for the second equality.  On the other hand, letting
$\mu_\ve = \vf_\ve \bmu$, we have $\mu_\ve \in \cB$ by Lemma~\ref{lem:mult}, and
\[
(\Pi_0\mu_\ve)(\psi) = \lim_{n\to \infty} \frac 1n \sum_{k=0}^{n-1} \Lp^n\mu_\ve(\psi)
= \lim_{n\to \infty} \frac 1n \sum_{k=0}^{n-1} \bmu(\vf_\ve \, \psi \circ T^k)
= \bmu(\vf_\ve \, \psi) + \mathcal{O}(\ve|\psi|_\infty)  .
\]
Thus, since $\mu_A(\cdot) := \bmu(\bo_A \, \cdot \, )$ can be approximated by elements of $\bV_0$, it
must belong to $\bV_0$.  Since this is true for each invariant set $A$, and $\bV_0$ is finite dimensional,
there can be only finitely many invariant sets of positive $\bmu$ measure.  Let $\{ A_i \}$ 
denote the minimal (finite) partition of $M$ into invariant sets of positive $\bmu$ measure. It follows that 
any $\mu \in \bV_0$ can be written as a linear combination of the measures $\mu_{A_i}$.
\end{proof}

Many other properties regarding the characterization of physical measures and the ergodic decomposition
with respect to Lebesgue measure can be derived in this general setting.  We refer the interested
reader to \cite[Lemma~5.7]{demers liverani}.  
As an example, we present one such result here, whose proof is similar to that of Lemma~\ref{lem:spec}(iii).

\begin{lemma}
\label{lem:physical}
$T$ admits only finitely many physical measures and they belong to $\mathbb{V}_0$.
\end{lemma}

\begin{proof}
Suppose $\mu$ is a physical measure and let $B \subset M$ with $m(B)>0$ be such that
\[
\lim_{n \to \infty} \frac 1n \sum_{k=0}^{n-1} \psi \circ T^k(x) = \mu(\psi) , \qquad \forall \; x \in B, \; \psi \in C^0(M).
\]
Let $\ve > 0$ and take $\vf_\ve \in C^1(M)$ such that $m(|\bo_B - \vf_\ve|) \le \ve$. Define $m_\ve = \vf_\ve m$.
Then for $\psi \in C^0(M)$,
\[
\begin{split}
\Pi_0 m_\ve(\psi) & = \lim_{n \to \infty} \frac 1n \sum_{k=0}^{n-1} \Lp^k m_\ve(\psi) = \lim_{n \to \infty} \frac 1n
\sum_{k=0}^{n-1} m_\ve(\psi \circ T^k) \\
& =  \lim_{n \to \infty} \frac 1n \sum_{k=0}^{n-1} m(\bo_B \psi \circ T^k) + \mathcal{O}(\ve |\psi|_{C^0}) 
\; =  \; m(B) \mu(\psi) + \mathcal{O}(\ve |\psi|_{C^0}) \, .
\end{split}
\]
Since $\Pi_0m_\ve \in \mathbb{V}_0$, we see that $\mu$ can be approximated by elements of $\mathbb{V}_0$, and
therefore belongs to $\mathbb{V}_0$.  Since $\mathbb{V}_0$ is finite dimensional and a physical measure is necessarily
ergodic, there can be only finitely many physical measures for $T$.
\end{proof}

Since our Baker's transformation $T$
is mixing, we proceed to prove that $\Lp$ has a spectral gap.

\begin{lemma}
\label{lem:one}
$T$ has a single ergodic component with positive $\bmu$ measure.
\end{lemma}

\begin{proof}
Let $A = A_i$ be an invariant set with positive $\bmu$ measure on which $(T, \bmu)$ is ergodic. 
For $x \in M$, let $U(x)$ denote the
element of $\cW^u$ containing $x$.  Define $A_u = \cup_{x \in A} U(x)$.  
Clearly, $T^{-1}A_u \subset A_u$.  We claim that
$\bmu(A_u \setminus A) = 0$.

To see this, for each $\psi \in C^0(M)$, the backward ergodic average of $\psi$ on $A_u$ equals the ergodic average
of $\psi$ on $A$.  Thus for $x \in A_u$,
\begin{equation}
\label{eq:ergodic}
\frac 1n \sum_{k=0}^{n-1} \psi \circ T^{-k}(x) \to \frac{\bmu(\psi \bo_A)}{\bmu(A)} .
\end{equation}
Now $\bmu(\psi \circ T^{-k} \, \bo_{A_u}) \le \bmu(\psi \, \bo_{A_u})$
due to the inclusion $T^{-1}A_u \subset A_u$.  Thus integrating over both sides of \eqref{eq:ergodic}
yields $\bmu(\psi \bo_{A_u}) \ge \frac{\bmu(\psi \bo_A)}{\bmu(A)} \bmu(A_u)$.
For each $\ve >  0$, we choose $\psi_\ve \in C^0(M)$ such that 
$\bmu(|\bo_A - \psi_\ve|) < \ve$.  Then the above inequality yields $\bmu(A) \ge \bmu(A_u) + \mathcal{O}(\ve)$, which proves
the claim since $A \subseteq A_u$.

Next, let $A_u^s = \cup_{x \in A_u} W(x)$, where $W(x)$ is the element of $\cW^s$ containing $x$.  Due to the product structure
of $\cW^s$ and $\cW^u$, it is clear that $A_u^s = M$.  Repeating the argument above, we can show that
$\bmu(A_u^s \setminus A_u) = 0$ (using the forward ergodic average, and integrating over $A_u^s$ rather than $A_u$).
Thus $\bmu(M \setminus A) = 0$, so that $T$ has a single ergodic component with positive $\bmu$ measure.
\end{proof}

\begin{cor}
\label{cor:spectral gap}
$\Lp$ has a spectral gap on $\cB$.
\end{cor}

\begin{proof}
From Lemma~\ref{lem:one}, we see that the eigenspace $\bV_0$ is one-dimensional for any $T = T_{\kappa, \lambda}$.
Now suppose the peripheral spectrum contains an eigenvalue $e^{2 \pi i \theta}$, and let $\mu \in \bV_\theta$ be a probability 
measure.    By Lemma~\ref{lem:spec}(ii),
$\theta = p/q$, for some $p, q \in \mathbb{N}$.  
Thus $\Lp^q \mu = \mu$, and so $\mu$ is an invariant measure for $T^q$.  But $T_{\kappa, \lambda}^q = T_{\kappa^q, \lambda^q}$
is simply another Baker's map of the type described in Section~\ref{bake}, and so Lemmas~\ref{lem:spec} and \ref{lem:one}
apply to $T^q$ as well. It follows that $\mu = \bmu$ and $\theta = 0$, so 1 is the only element of the peripheral spectrum for $\Lp$,
and it is a simple eigenvalue.  The quasi-compactness of $\Lp$ implies that the spectrum of $\Lp$
outside the disk of radius $\rho$ is finite, and so
$\Lp$ has a spectral gap.
\end{proof}

\begin{proof}[Proof of Theorem~\ref{thm:spec}]
The first statement and item (a) of the theorem are Proposition~\ref{prop:quasi} and 
Corollary~\ref{cor:spectral gap}.  

For item (b), Lemma~\ref{lem:spec} proves that $\bmu$ is a measure.  
The fact that its conditional measures on unstable leaves are equal to arclength 
follows by applying \eqref{eq:strong unstable}
to $\bmu$, which implies $\| \bmu \|_u = 0$.  Furthermore, letting $K(\cW^u)$ denote
the set of bounded, measurable functions on $M$ which are constant on unstable leaves,
we note that $\Lp(K(\cW^u)) \subset K(\cW^u)$ and $K(\cW^u) \subset C^1(\cW^u) \subset \cB$.
Thus $\bmu$ lies in the closure of this set in the $\| \cdot \|_{\cB}$ norm and so has constant conditional
densities on unstable leaves.
Finally, the claim about $\mu_0$ being the unique physical measure for the system follows from
Lemma~\ref{lem:physical} and Lemma~\ref{lem:one}.

The convergence for item (c) follows immediately from the decomposition \eqref{eq:decomp}
and the existence of a spectral gap.
\end{proof}


\subsection{Proofs of limit theorems}
\label{limit proofs}

In this section, we outline how Theorem~\ref{thm:limit} follows from the established
spectral picture using standard arguments.  
Let $g \in C^1(M)$ and define $S_ng = \sum_{j =0}^{n-1} g \circ T^j$.
We define the generalized transfer operator $\Lp_g$ on $\cB$ by,
$\Lp_g f(\psi) = f (e^g \psi \circ T)$ for all $f \in \cB$.  It is then immediate that
\[
\Lp_g^n f(\psi) = f(e^{S_ng} \psi \circ T^n), \;\;\; \mbox{for all $n \ge 0$.}
\]
The main element in the proofs of the limit theorems is that $\Lp_{zg}$, $z \in \mathbb{C}$, is
an analytic perturbation of $\Lp = \Lp_0$ for small $|z|$.
\begin{lemma}
\label{lem:analytic}
For $g \in C^1(M)$, the map $ z \mapsto \Lp_{zg}$ is analytic for all $z \in \mathbb{C}$.
\end{lemma}

\begin{proof}
The lemma will follow once we show that
our strong norm $\| \cdot \|_{\cB}$ is continuous with respect to multiplication by $e^{zg}$.
This is straightforward using Lemma~\ref{lem:mult}.

Define the operator $\pa_n f = \Lp(g^n f)$, for $f \in \cB$.  Then Lemma~\ref{lem:mult} implies
\[
\| \pa_n(f) \|_{\cB} = \| \Lp (g^n f ) \|_{\cB} \le 3 \| \Lp \|  \, \| f \|_{\cB} \, |g^n|_{C^1(M)},
\]
and $|g^n|_{C^1(M)} \le (|g|_{C^1(M)})^n$.
This allows us to conclude that the operator 
$\sum_{n=0}^\infty \frac{z^n}{n!} \pa_n$ is well-defined on $\cB$ and equals $\Lp_{zg}$,
\[
\sum_{n=0}^\infty \frac{z^n}{n!} \pa_nf (\psi) = f\left( \sum_{n=0}^\infty \frac{z^n}{n!} g^n \cdot \psi \circ T \right) = f (e^{zg} \psi) = \Lp_{zg}f(\psi) , \;\; \; \mbox{for } \psi \in C^1(\cW^s) ,
\] 
since we know the sum converges in the $\| \cdot \|_{\cB}$ norm.
\end{proof}

With the analyticity of $z \mapsto \Lp_{zg}$ established, it follows from analytic
perturbation theory \cite{kato} that both the discrete
spectrum and the corresponding spectral projectors of $\Lp_{zg}$ vary smoothly with $z$.  Thus, since $\Lp_0$ has
a spectral gap, then so does $\Lp_{zg}$ for $z \in \mathbb{C}$ sufficiently close to 0.

At this point, the proofs of Theorem~\ref{thm:limit}(a), (b) and (c) can proceed verbatim as in
\cite[Section 6]{dz1} with only slight modifications for notation.


\section{Singular Limit $\lambda \to 0$}
\label{singular}

In this section, we prove the results stated in Section~\ref{sec:sing}.  We recall the notation
introduced there for a sequence of Baker's maps $(T_\lambda)_{\lambda \in [0,\lambda_0]}$
such that $\cap_{\lambda \in [0, \lambda_0]} T_\lambda(R_i) = U_i$ for a fixed
choice of $U_i$, $i = 1, \ldots, \kappa$.  For the remainder of this section, we fix such a family
of maps and prove the relevant results.


\subsection{Standard pairs and averages}
\label{sec:pairs}

For any $f \in \cB$ and $W\in \cW^s$, the distribution $D^1_W(f)$ as defined in \eqref{eq:distribution} exists
and is well-defined.  Thus we may define a function $\of$ by
\begin{equation}
\label{eq:avg dist}
\of(x) = \langle D^1_{W_x}(f), \bo \rangle = \int_{W_x} f \, dm_{W_x}, 
\end{equation}
where $W_x$ is the element of $\cW^s$ containing $x$, and the integral representation is valid for any $f \in C^1(\cW^u)$.  
Since $\of$ is constant on stable leaves, we also denote
$\of(x)$ by $\of_W$, when convenient, where $W = W_x$. 

\begin{lemma}
\label{lem:avg}
If $f \in \cB$, then $\of \in C^\beta(\cW^u)$ and 
\[
|\of|_{C^0(\cW^u)} \le |f|_w \le \| f \|_s, \qquad
\sup_{U \in \cW^u} H^\beta_U(\of) \le \| f \|_u .
\]  
Similarly, if $f \in C^1(\cW^u)$, then $\of \in C^1(\cW^u)$ and $|\of|_{C^1(\cW^u)} \le |f|_{C^1(\cW^u)}$.
\end{lemma}

\begin{proof}
As usual, by density, it suffices to prove the first statement for $f \in C^1(\cW^u)$.  Since $\of$ is constant on stable leaves,
it suffices to check the $C^1$ norm on a single $U \in \cW^u$. 

Fix $U \in \cW^u$ and let $x \in U$.  Then,
\[
|\of(x)| = \left| \int_{W_x} f \, dm_{W_x} \right| \le | f |_w |\bo|_{C^1(W_x)} = |f|_w . 
\]
Thus $\sup_{x \in U} |\of(x)| \le |f|_w \le \| f \|_s$.  Next, if $x, y\in U$, then,
\[
|\of(x) - \of(y)| =\left| \int_{W_x} f \, dm_{W_x} - \int_{W_y} f \, dm_{W_y} \right| \le \| f \|_u |x-y|^\beta  .
\]
Putting these estimates together proves the first statement of the lemma.

The second statement is proved similarly, using $\sup_{x \in M} |f(x)|$ in place of $\| f\|_s$ and $H^1_U(f)$ in place of
$\| f \|_u$.
\end{proof}

Recall that a standard pair is a measure supported on a single $U \in \cW^u$ together with a smooth density $f \in C^1(\cW^u)$, which we denote by $f \delta_U$.
We next prove Lemma~\ref{lem:stand}, that standard pairs belong to $\cB$.

\begin{proof}[Proof of Lemma~\ref{lem:stand}]
Let $U \in \cW^u$ and $f \in C^1(U)$.  To show that $f\delta_U \in \cB$, we must approximate $f\delta_U$ in the $\cB$ norm
by elements of $C^1(\cW^u)$ viewed as densities with respect to Lebesgue measure.

For $\ve > 0$, let $N_{\ve/2}(U)$ denote the $\ve/2$ neighborhood of $U$ in $M$.  $N_{\ve/2}(U)$ is a horizontal strip of width
$\ve$.  We extend $f$ to $N_{\ve/2}$ as follows:  Define $\of_\ve$ to be a function constant on stable leaves in 
$N_{\ve/2}(U)$ such that $\of_\ve(x) = \ve^{-1} f(x)$ for $x \in U$.  On $M\setminus N_{\ve/2}(U)$, 
set $\of_\ve = 0$.
It follows that $|\of_\ve|_{C^1(\cW^u)} = \ve^{-1} |f|_{C^1(U)}$.

We will show that $(\of_\ve)_{\ve > 0}$ forms a Cauchy sequence in $\cB$.  Let $0 < \ve_1 < \ve_2$.  Suppose
$W \in \cW^s$ and $\vf \in C^\alpha(W)$ with $|\vf|_{C^\alpha(W)} \le 1$.  Define $\vf_U = \vf(W \cap U)$.  Then,
\[
\int_W (\of_{\ve_1} - \of_{\ve_2}) \vf \, dm_W = \int_W (\of_{\ve_1} - \of_{\ve_2}) (\vf - \vf_U) \, dm_W
+ \int_W (\of_{\ve_1} - \of_{\ve_2}) \vf_U \, dm_W . 
\] 
The second integral above is 0 since $\vf_U$ is a constant and $\int_W \of_{\ve_1} \, dm_W = \int_W \of_{\ve_2} \, dm_W = f(W \cap U)$.

To estimate the first integral above, we use that,
\[
\sup \{ |\vf(x) - \vf_U| : x \in W \cap N_{\ve}(U) \} \le \ve^\alpha H^\alpha_W(\vf) , \quad \forall \, \ve > 0 .
\]
We split up the integral into two parts, using the fact that $f_{\ve_1} - f_{\ve_2} \ge 0$ on $W_{\ve_1} := W \cap N_{\ve_1/2}(U)$.
Similarly, define $W_{\ve_2} = W \cap N_{\ve_2/2}(U)$.
\[
\begin{split}
\int_W (\of_{\ve_1} & - \of_{\ve_2}) (\vf - \vf_U) \, dm_W
= \int_{W_{\ve_1}} (\of_{\ve_1} - \of_{\ve_2}) (\vf - \vf_U) \, dm_W - \int_{W_{\ve_2}\setminus W_{\ve_1}} \of_{\ve_2} \, (\vf - \vf_U) \, dm_W \\
& \le \Big( \frac{\ve_1}{2} \Big)^\alpha |f|_{C^0(U)} \int_{-\ve_1/2}^{\ve_1/2} ( \ve_1^{-1} - \ve_2^{-1}) \, dt
+ 2 \Big( \frac{\ve_2}{2} \Big)^\alpha |f|_{C^0(U)} \int_{\ve_1/2}^{\ve_2/2} \ve_2^{-1} \, dt \\
& \le \Big( \frac{\ve_1}{2} \Big)^\alpha |f|_{C^0(U)} \Big(1 - \frac{\ve_1}{\ve_2} \Big)
+ \Big( \frac{\ve_2}{2} \Big)^\alpha |f|_{C^0(U)} \Big(1 - \frac{\ve_1}{\ve_2} \Big)
\le 2 \Big( \frac{\ve_2}{2} \Big)^\alpha |f|_{C^0(U)} .
\end{split}
\]
Taking the appropriate suprema yields, $\| \of_{\ve_1} - \of_{\ve_2} \|_s \le 2 (\frac{\ve_2}{2})^\alpha |f|_{C^0(U)}$.

Next, let $W_i \in \cW^s$, $\vf_i \in C^1(W_i)$ with $|\vf_i|_{C^1(W_i)} \le 1$, $i =1,2$, and
$d_0(\vf_1, \vf_2) = 0$.  On the one hand, by the previous estimate with $\alpha = 1$, we have
\begin{equation}
\label{eq:first}
\left| \int_{W_1} (\of_{\ve_1} - \of_{\ve_2}) \vf_1 \, dm_{W_1} - \int_{W_1} (\of_{\ve_1} - \of_{\ve_2}) \vf_1 \, dm_{W_1} \right|
\le 2 \ve_2 |f|_{C^0(U)} .
\end{equation}
On the other hand, letting $\Theta : W_2 \to W_1$ denote the holonomy from $W_2$ to $W_1$, we note that
$\vf_1 \circ \Theta = \vf_2$ since $d_0(\vf_1, \vf_2) = 0$.  Thus,
\[
\int_{W_1} \of_{\ve_1} \vf_1 \, dm_{W_1} - \int_{W_2} \of_{\ve_1} \vf_2 \, dm_{W_2}
= \int_{W_2} (\of_{\ve_1} \circ \Theta - \of_{\ve_1}) \, \vf_2 \, dm_{W_2}
\le |\vf_2|_{C^0(W_2)} d(W_1, W_2) H^1_U(f) .
\]
A similar estimate holds for $\of_{\ve_2}$, so that
\[
\left| \int_{W_1} (\of_{\ve_1} - \of_{\ve_2}) \vf_1 \, dm_{W_1} - \int_{W_1} (\of_{\ve_1} - \of_{\ve_2}) \vf_1 \, dm_{W_1} \right|
\le 2 d(W_1, W_2) H^1_U(f) .
\]
Combining this estimate with \eqref{eq:first} and interpolating, we have that the difference in the integrals is bounded by
\[
2 H^1_U(f) \min \{ \ve_2, d(W_1, W_2) \} \le 2H^1_U(f) \ve_2^{1-\beta} d(W_1, W_2)^\beta .
\]
Dividing through by $d(W_1, W_2)^\beta$ and taking the appropriate suprema yields
$\| \of_{\ve_1} - \of_{\ve_2} \|_u \le 2 H^1_U(f) \ve_2^{1-\beta}$.

Thus $(\of_\ve)_{\ve >0}$ forms a Cauchy sequence in $\cB$ and therefore converges.  Moreover, once it exists, this limit
must coincide with the limit of the sequence in $(C^1(\cW^s))^*$, due to injectivity (Lemma~\ref{lem:embeddings}),
and is thus the measure $f\delta_U$. 
\end{proof}


\subsection{Quasi-compactness of the transfer operator $\Lp_0$:  Proof of Theorem~\ref{thm:L0}}
\label{quasi L0}

Recall that the map $T_0$ satisfies $T_0(M) = \cup_{i=1}^\kappa U_i = \oM_0$.   
According to Lemmas~\ref{lem:sing L} and \ref{lem:avg}, 
for $f \in \cB$, the transfer operator $\Lp_0$ is defined by \eqref{eq:L0 def}, and its
iterates by
\[
\Lp_0^n f = \sum_{i=1}^\kappa \kappa^{-1} \bar g \delta_{U_i} \, , \quad \mbox{where} \quad
\bar g(x) = \kappa^{-n+1} \sum_{T_0^nW = x} \of_W , \quad \mbox{for $x \in U_i$,}
\]
and  $\of(x) = \of_{W_x}$ is defined by \eqref{eq:avg dist}. The following lemma
provides the crucial Lasota-Yorke inequalities for $\Lp_0$.

\begin{lemma}
\label{lem:L0 bounded}
$\Lp_0$ is a bounded operator on $\cB$.  Moreover, for each $f \in \cB$, and $n \ge 1$,  we have
\begin{eqnarray}
  | \Lp_0^n f |_w  & \le & |f|_w,   \label{eq:L0 weak} \\
  \| \Lp_0^n f \|_s & \le &  |f|_w, \label{eq:L0 stable}  \\ 
  \| \Lp_0^n f \|_u & \le & \kappa^{-\beta n} \| f \|_u.  \label{eq:L0 unstable} 
\end{eqnarray}
\end{lemma}

\begin{proof}
From \eqref{eq:L0 def} and Lemma~\ref{lem:stand}, it follows that if $f \in C^1(\cW^u)$, 
then $\Lp_0f \in \cB$.  To extend
this to the completion, we need to show that $\Lp_0$ is continuous in the $\cB$ norm.  This follows from the Lasota-Yorke
inequalities, which we now proceed to prove.

\medskip
\noindent
{\em Proof of Weak and Strong Stable Norm Estimates.}  Let $f \in \cB$, $W \in \cW^s$ and $\vf \in C^\alpha(W)$
such that $|\vf|_{C^\alpha(W)} \le 1$.  Then for $n \ge 1$,
\[
\begin{split}
\int_W \Lp_0^n f \, \vf \, dm_W & = \sum_{i=1}^\kappa \; \sum_{T_0^n W_j = W \cap U_i} \kappa^{-n} \int_{W_j} f \, \vf \circ T_0^n \, dm_{W_j}
\\
& \le \sum_{i=1}^\kappa \; \sum_{T_0^n W_j = W \cap U_i} \kappa^{-n} \of_{W_j} \, \vf (U_i \cap W),
\end{split}
\]
where we have used the fact that $\vf \circ T_0^n$ is constant on $W_j$ since $T_0^n(W_j) = W \cap U_i$.  For each $i$, the cardinality
of $\{ W_j \in \cG_n(W) : T_0^nW_j = W \cap U_i \}$ is $\kappa^{n-1}$.  Also, $|\of_{W_j}| \le |f|_w$ by Lemma~\ref{lem:avg}.
Combining these facts yields,
\[
\int_W \Lp_0^n f \, \vf \, dm_W \le |f|_w .
\]
Now taking the supremum over $W \in \cW^s$ and $\vf \in C^\alpha(W)$ yields \eqref{eq:L0 stable}.  On the other hand, replacing
$\vf \in C^\alpha(W)$ by $\vf \in C^1(W)$ in the estimates above yields \eqref{eq:L0 weak}.

\medskip
\noindent
{\em Proof of Strong Unstable Norm Estimate.}
Take again $f \in \cB$, $W^k \in \cW^s$,  $\vf_k \in C^1(W^k)$, with $|\vf_k|_{C^1(W^k)} \le 1$, $k=1,2$, and $d_0(\vf_1, \vf_2) =0$.
Then for $n \ge 1$, denoting by $\{ W^k_j \}_j = \cG_n(W^k)$ the set of stable leaves such that $T_0^nW^k_j = W^k \cap U_i$
for some $i$, we estimate
\[
\begin{split}
\int_{W^1} \Lp_0^n f \, \vf_1 \, dm_{W_1} & - \int_{W_2} \Lp_0^n f \, \vf_2 \, dm_{W_2} \\
& = \kappa^{-n} \sum_{i=1}^\kappa \; \sum_{T_0^nW_j^1 = W^1 \cap U_i} \int_{W_j^1} f \, \vf_1 \circ T_0^n \, dm_{W^1_j} - \int_{W_j^2} f \, \vf_2 \circ T_0^n \, dm_{W^2_j} \\
& = \kappa^{-n} \sum_{i=1}^\kappa \; \sum_{T_0^nW_j^1 = W^1 \cap U_i} \vf_1(W^1 \cap U_i) \left( \int_{W_j^1} f \, dm_{W^1_j} - \int_{W_j^2} f \,  dm_{W^2_j} \right) \\
& \le \kappa^{-\beta n} d(W^1, W^2)^\beta \| f \|_u ,
\end{split}
\]
where we have used the fact that $\vf_1(W^1 \cap U_i) = \vf_2(W^2 \cap U_i)$ since $d_0(\vf_1, \vf_2) =0$.
Dividing through by $d(W_1, W_2)^\beta$ and taking the appropriate suprema yields the required estimate.
\end{proof}

\begin{cor}
\label{cor:L0}
$\Lp_0$ is quasi-compact as an operator on $\cB$.  Its spectral radius equals 1, and its essential spectral radius is at most
$\kappa^{-\beta}$.
\end{cor}

\begin{proof}
That the essential spectral radius is at most $\kappa^{-\beta}$ follows from the estimate,
\[
\| \Lp_0^n f \|_{\cB} \le \kappa^{-\beta n} \| f\|_u + |f|_w \le \kappa^{-\beta n} \| f \|_{\cB} + |f|_w,
\] 
and the standard Hennion argument \cite{hennion spec}.  Replacing $|f|_w$ above by $\|f \|_s$ 
yields $\| \Lp_0^n f \|_{\cB} \le \|f \|_{\cB}$,
which implies that the spectral radius is at most 1.  To see that the spectral radius is in fact 1, note that 
$\mu_0 := \kappa^{-1} \sum_{i=1}^\kappa \delta_{U_i}$ belongs to $\cB$ by Lemma~\ref{lem:stand}, and 
$\Lp_0 \mu_0 = \mu_0$ since $\overline{(\delta_{U_i})} = \bo$ on $M$ for each $i$,
where $\overline{(\delta_{U_i})}$ is defined by \eqref{eq:avg dist}.
\end{proof}

\begin{proof}[Proof of Theorem~\ref{thm:L0}]
Lemma~\ref{lem:L0 bounded} and Corollary~\ref{cor:L0} together prove the first two statements of Theorem~\ref{thm:L0}.

To prove that $\Lp_0$ has a spectral gap, we could proceed to prove a spectral decomposition as in Lemma~\ref{lem:spec};
however, in the case $\lambda=0$, we have a simpler argument at our disposal. 

Let $\oT_0$ be the restriction of $T_0$ to $\oM_0 = \cup_{i=1}^\kappa U_i$. It is clear that 
$\oT_0$ is a uniformly expanding map.  Indeend $\{ U_i \}_{i=1}^\kappa$
constitutes a finite Markov partition for $\oT_0$, and $\oT_0$ is full-branched map on this partition:  $\oT_0(U_i) = \oM_0$ for
each $i$.  Let $\overline{\Lp}_0$ denote the transfer operator corresponding to $\oT_0$.  It is by now a classical result
that $\overline{\Lp}_0$ has a spectral gap on $C^\beta(\oM_0)$ (see, for example \cite[Chapter 5]{ruelle}).

It is clear from
\eqref{eq:L0 def} that restricted to $\cup_{i=1}^\kappa U_i$, and $\Lp_0 = \overline{\Lp}_0$.  Moreover, by 
Lemma~\ref{lem:avg} it follows that $\Lp_0(\cB)$ can be identified with a subset of $C^\beta(\oM_0)$.  
Thus any element of the peripheral
spectrum of $\Lp_0$, must belong to $C^\beta(\oM_0)$.  Since $\overline{\Lp}_0$ has no other eigenvalues aside from
1 on the unit circle, and 1 is a simple eigenvalue, this must also be true of $\Lp_0$; and by quasi-compactness,
$\Lp_0$ has a spectral gap on $\cB$.  

The formula for the unique invariant probability measure $\mu_0$
follows by inspection as in the proof of Corollary~\ref{cor:L0}.  Finally, the exponential convergence to $\mu_0$
follows from the spectral gap.
\end{proof}


\subsection{Proof of Theorem~\ref{thm:conv}}
\label{perturb}

Recall the norm, defined in Section~\ref{sec:sing}, for an operator $\Lp: \cB \to \cB_w$,
\[
||| \Lp ||| = \sup\{ |\Lp f |_w : \| f \|_{\cB} \le 1 \} .
\] 
In this section, we will prove Theorem~\ref{thm:conv}, whose main estimate is the bound on 
$||| \Lp_0 - \Lp_\lambda  |||$ stated in \eqref{eq:close}.

\begin{proof}[Proof of Theorem~\ref{thm:conv}]
Let $f \in \cB$, $W \in \cW^s$ and $\vf \in C^1(W)$ such that $| \vf |_{C^1(W)} \le 1$.  Note that 
$W' \in \cW^s$ satisfies $T_0 W' \subset W$ if and only if $T_\lambda W' \subset W$ for all $\lambda \le \lambda_0$.  
Moreover, for each $i=1, \ldots \kappa$, there is a unique $W_i \in\cW^s$ such that $T_0W_i = W \cap U_i$. 
We estimate,
\[
\begin{split}
\int_W (\Lp_0 - \Lp_\lambda) f \, \vf \, dm_W & = \sum_{i=1}^\kappa \kappa^{-1} \int_{W_i} f \, (\vf \circ T_0 - \vf \circ T_\lambda) \, dm_{W_i} \\
& \le \sum_{i=1}^\kappa \kappa^{-1} \| f \|_s |\vf \circ T_0 -  \vf \circ T_\lambda|_{C^\alpha(W_i)} .
\end{split}
\]
For $x \in W_i$, we have $T_0(x) = W \cap U_i$ and $T_\lambda(x) \subset W \cap N_{\lambda}(U_i)$, thus,
\[
|\vf \circ T_0(x) - \vf \circ T_\lambda(x)| \le H^1_W(\vf) |T_0(x) - T_\lambda(x)| \le H^1_W(\vf) \lambda .
\]
Now for $x,y \in W_i$, since $\vf \circ T_0(x) = \vf \circ T_0(y)$, we have
\[
\begin{split}
|(\vf \circ T_0 - \vf \circ T_\lambda)(x) - (\vf \circ T_0 - \vf \circ T_\lambda)(y)| 
& = |\vf \circ T_\lambda(x) - \vf \circ T_\lambda(y)|
\le H^1_W(\vf) |x-y| \\
& \le H^1_W(\vf) |x-y|^\alpha \lambda^{1-\alpha}.
\end{split}
\]
Putting these estimates together, we conclude $|\vf \circ T_0 - \vf \circ T_\lambda|_{C^\alpha(W_i)} \le \lambda^{1-\alpha} |\vf|_{C^1(W)}$.
Thus,
\[
\int_W (\Lp_0 - \Lp_\lambda) f \, \vf \, dm_W \le \| f \|_s \lambda^{1-\alpha},
\]
and taking the appropriate suprema proves \eqref{eq:close}.

The claim regarding the spectra and spectral projectors of $\Lp_0$ and $\Lp_{\lambda}$ follow from
\cite[Corollary 1]{KL} and the uniform (in $\lambda$) Lasota-Yorke inequalities given by
Proposition~\ref{prop:ly} and Lemma~\ref{lem:L0 bounded}.  Since $\Lp_0$ has a spectral gap, the convergence of the
spectral projectors also implies the convergence of the invariant measures $\mu_\lambda$, as claimed.
\end{proof}


\end{document}